\documentclass[pdflatex,sn-mathphys-num]{sn-jnl}% Math and Physical Sciences Numbered Reference Style
\usepackage{amssymb}
\usepackage{amsmath}
\usepackage[utf8]{inputenc}
\usepackage{latexsym}
\usepackage{amssymb}
\usepackage{amsmath}
\usepackage{amsthm}
\usepackage{amsfonts}

\usepackage{soul, color}
\usepackage{dsfont}
\usepackage{verbatim}
\usepackage{graphicx}
\usepackage{graphics}
\usepackage{enumerate}
\usepackage{url}
\usepackage{mathtools}
\usepackage{comment}
\usepackage{bbm}
\usepackage{hyperref}
\usepackage{relsize}
\usepackage{subcaption}
\usepackage{placeins}
\allowdisplaybreaks

\usepackage[capitalize]{cleveref}
\crefname{equation}{}{}

\def\ev{{\mathrm{e}}}

\def\lb{{\Delta_\Gamma}}

\def\del{{\partial}}

\def\gj{{G_\Gamma^j}}
\def\G{{\mathcal{G}}}
\def\L{{\mathcal{L}}}
\def\k{{k_\Gamma}}
\def\o{{\circ}}
\def\K{{\mathcal{K}}}
\def\tK{{\tilde{\mathcal{K}}}}
\newcommand\tosp[2]{{{#2},{#1}}}
\newcommand{\tp}{{\scriptscriptstyle\mathsf{T}}}
\def\cS{{\mathcal{S}}}
\def\cD{{\mathcal{D}}}

\newtheorem{theorem}{Theorem}[section]
\newtheorem{corollary}{Corollary}[theorem]
\newtheorem{lemma}[theorem]{Lemma}
\theoremstyle{definition}

\theoremstyle{remark}
\newtheorem*{remark}{Remark}

\begin{document}

\title[A Boundary Integral Method for Generalized Impedance Boundary Conditions in 2D]{A Boundary Integral Method for Generalized Impedance Boundary Conditions in 2D}

\author*[1]{\fnm{Jacob} \sur{Linden}}\email{jacoblinden@uchicago.edu}

\author[2]{\fnm{Travis} \sur{Askham}}\email{askham@njit.edu}

\author[1,3]{\fnm{Jeremy} \sur{Hoskins}}\email{jeremyhoskins@uchicago.edu}

\affil[1]{\orgdiv{Computational and Applied Mathematics}, \orgname{University of Chicago}, \orgaddress{\street{5747 S Ellis Ave}, \city{Chicago}, \postcode{60637}, \state{IL}, \country{USA}}}

\affil[2]{\orgdiv{Department of Mathematical Sciences}, \orgname{New Jersey Institute of Technology}, \orgaddress{\street{70 Summit St}, \city{Newark}, \postcode{07103}, \state{NJ}, \country{USA}}}

\affil[3]{\orgname{NSF-Simons National Institute for Theory and Mathematics in Biology}, \orgaddress{\street{172 E Chesnut St}, \city{Chicago}, \postcode{60611}, \state{IL}, \country{USA}}}


\abstract{Generalized impedance boundary conditions for the Helmholtz equation 
arise in a number of acoustic and electromagnetic models that approximate 
boundary effects, including thin coatings and the viscous and thermal losses
that can occur in boundary layers.
Here we present a well-conditioned numerical method for this class of boundary conditions
based on a boundary integral re-formulation of the equations. The formulation applies a combined layer 
representation together with a surface preconditioner to obtain an integral 
equation of the second kind. Additionally, by using appropriate image sources, we construct well-conditioned 
representations of impedance-to-impedance maps for subdomains carrying these boundary conditions 
on part of their boundary, so that the method can be used within a domain-decomposition framework. 
Several numerical examples in geometries inspired by acoustic applications demonstrate 
the efficacy of the method.}

\keywords{boundary layer, visco-thermal, Helmholtz, generalized impedance, boundary integral equations, Fredholm second-kind, method of images, impedance-to-impedance, phase plug}



\maketitle

\section{Introduction}

%While wave equation (or Helmholtz equation, for time harmonic waves) and acoustic ray models are common in acoustic wave propagation, 
%the effects of the fluid medium must be taken into account to obtain an accurate model
%in certain regimes.
%In particular, viscous and thermal effects are significant in a variety of acoustic devices with narrow regions, including hearing aids \cite{kampinga_viscothermal_2010,cordioli_comparison_2010}, micro-electro-mechanical systems \cite{homentcovschi_microacoustic_2014,naderyan_computational_2019,hassanpour_guilvaiee_femmodeling_2023}, metamaterials \cite{henriquez_viscothermal_2017,moleron_visco-thermal_2016,romerogarcia_viscothermal_2019,joseph_impact_2024}, perforated panels \cite{carbajo_finite_2015,billard_numerical_2021,na_unified_2023}, and phase plugs \cite{berggren_highly_2018,andersen_shape_2019,christensen_compression_2011}.
%There is a large body of recent work on the design and optimization of such devices \cite{bernland_shape_2019,berggren_better_2024,mousavi_extending_2023,noguchi_topology_2022,dilgen_three_2022,tissot_optimal_2020,henriquez_three-dimensional_2018}, and, more generally, the simulation of visco-thermal effects \cite{sambuc_numerical_2014,kampinga_performance_2010,cops_estimation_2020,preuss_revising_2023,paltorp_open-source_2024,andersen_modelling_2018}. 
%Many of these studies model acoustic waves using the linearized Navier-Stokes equations, or simplifications
%thereof, and solve the equations by boundary element and finite element methods. 

We consider the Helmholtz equation, 
$$ \Delta u + k^2 u = 0 \; , $$ 
subject to generalized impedance boundary conditions (GIBCs) of the form 
\begin{equation}\label{eq:vtbc}
    c_1\lb u+c_2u+\del_n u = f \; ,
\end{equation}
where $\lb$ denotes the Laplace-Beltrami operator on the boundary, $\partial_n$ denotes the normal derivative, 
and $c_1$ and $c_2$ are constants depending on the frequency and other model parameters.
Such boundary conditions arise from asymptotic expansions of the solution in problems where 
certain physical effects are localized in a thin layer around the boundary. Helmholtz equations with GIBCs appear
in a number of models in electromagnetics~\cite{senior1995approximate,ammari1999generalized,durufle2006higher},
including models of thin dielectric coatings~\cite{bendali1996effect,aslanyurek2011generalized},
and acoustics, including models of ultrasonics~\cite{shumpert2000impedance}, strongly absorbing obstacles~\cite{antoine2001high,haddar2005generalized} and 
viscous and thermal losses~\cite{berggren_acoustic_2018}.
%These boundary conditions are also referred to as \emph{Wentzell} boundary conditions~\cite{venttsel_boundary_1959,coclite_role_2009}.
While we focus on the constant coefficient case here, some models feature variable coefficient analogues of the above
boundary conditions, e.g.~\cite{bendali1996effect,antoine2001high,aslanyurek2011generalized}, especially with some dependence
on the curvature. We expect much of the strategy to extend to that case as well. 

Many of the existing computational approaches to problems with boundary conditions like
\cref{eq:vtbc} apply a finite element discretization of the Helmholtz equation in the volume~\cite{haddar2005generalized,antoine2005approximation,berggren_acoustic_2018,caubet2024ventcel,aslanyurek2011generalized}.
While the finite element approach is quite general and there are well-developed
numerical tools, the corresponding discrete linear systems tend to be ill-conditioned
for fine meshes and the cost of re-meshing, if the application involves shape optimization
or an inverse obstacle problem, is more significant than it is for approaches restricted to the
boundary. 

In general, boundary integral methods~\cite{Atkinson1997,greengard_accelerating_1998,sauter2011boundary,colton_integral_2013,martinsson2019fast}
offer advantages for the solution of the Helmholtz equation, particularly
for problems in complicated geometries. These 
methods represent the solution throughout the domain as a layer potential induced by a density defined on the 
boundary alone, so that only the boundary of the domain needs to be discretized and 
fewer degrees of freedom can be used to resolve a given problem. In the case 
of a boundary condition like \eqref{eq:vtbc}, the boundary conditions on $u$ become
an integro-differential equation for the boundary density. Frequently, these integro-differential
equations, or some equivalent variational formulation, are discretized directly; see, e.g.,
\cite{bendali1998boundary,cakoni2012integral,kress2018integral,banjai2022time}.
While this approach reduces the dimensionality of the problem, the corresponding
linear systems tend to be ill-conditioned because of the surface Laplacian appearing 
in \cref{eq:vtbc}, and the formulations often involve hyper-singular integral operators.
For the case $c_1=0$, well-conditioned integral formulations are known,
with \cite{turc2017well} providing a robust formulation using a regularized combined
field integral equation. Elliptic surface
partial differential equations, without a corresponding volume equation, have also been treated both 
analytically and numerically by integral equation methods, with \cite{gemmrich2008boundary,oneil_second-kind_2018,goodwill2025parametrix} being
recent numerical examples.

In this work, we construct and analyze a novel boundary integral equation formulation of the Helmholtz equation with the
boundary condition~\eqref{eq:vtbc}. The method is based on
a particular combined field representation of $u$ and an analytic surface pre-conditioner,
resulting in an integral equation for the density that is Fredholm second-kind with weakly
singular integral kernels. To the best of our knowledge, this is the first representation for 
the boundary condition \cref{eq:vtbc} with these
features. We show that this representation can also be modified to
handle ``domain decomposition'' problems, as in~\cite{berggren_acoustic_2018}, where the boundary
layer condition~\eqref{eq:vtbc} is imposed in a wave guide and impedance conditions are imposed on
artificial curves at the ends of the channel. Using a generalization of the method of images, it is possible to again obtain
an integral equation that is Fredholm second kind. This image technique may be of independent interest.
The resulting representations are amenable to standard integral equation methods for their discretization and efficient solution. We provide several numerical simulations on a variety of geometries, including one inspired by the cross section of a phase plug.

\section{Background}

\subsection{Description of the geometry}\label{subsec:geomdesc}

Consider a region $\Omega$ (not necessarily bounded) with positively oriented boundary $\Gamma \vcentcolon= \partial\Omega$ of finite length. In this paper we treat the following two classes of problems.

\textbf{Case I:} The boundary $\Gamma$ is smooth, and boundary conditions of the form \cref{eq:vtbc} are imposed on the entire boundary. In this case, we denote $\Gamma=\Gamma_*$. An example is given in Figure \ref{fig:casei}.

\textbf{Case II:} The boundary $\Gamma$ is the union of two collections of smooth, disjoint planar curves $\Gamma_*$ and $\Gamma_\o.$ We impose the additional assumptions that $\Gamma_*$ and $\Gamma_\o$ intersect at right angles, and that the curvature of $\Gamma_\o$ vanishes at its boundary. Generalized impedance boundary
conditions, \cref{eq:vtbc}, are imposed on $\Gamma_*$ and Robin boundary conditions are imposed on $\Gamma_\o$. Though in principle $\Gamma_\o$ could represent a physical surface, our main motivation is domain decomposition, where the solutions in adjacent regions are coupled together using impedance-to-impedance maps (see section \ref{subsec:dd}). In this context, the freedom to choose $\Gamma_\circ$ enables our imposition of the additional constraints.

For simplicity we assume in the sequel that $\Gamma_*$ and $\Gamma_\o$ are each the union of two disjoint curves: $\Gamma_*= \Gamma_*^1\cup\Gamma_*^2$ and $\Gamma_\o= \Gamma_\o^1\cup\Gamma_\o^2$. However, our method easily extends to treat arbitrary (finite) numbers of connected components. An example geometry is given in Figure \ref{fig:caseii}.

\begin{figure}[htbp]
    \centering
    \begin{subfigure}[t]{0.8\textwidth}
        \centering
        \includegraphics[width=\textwidth]{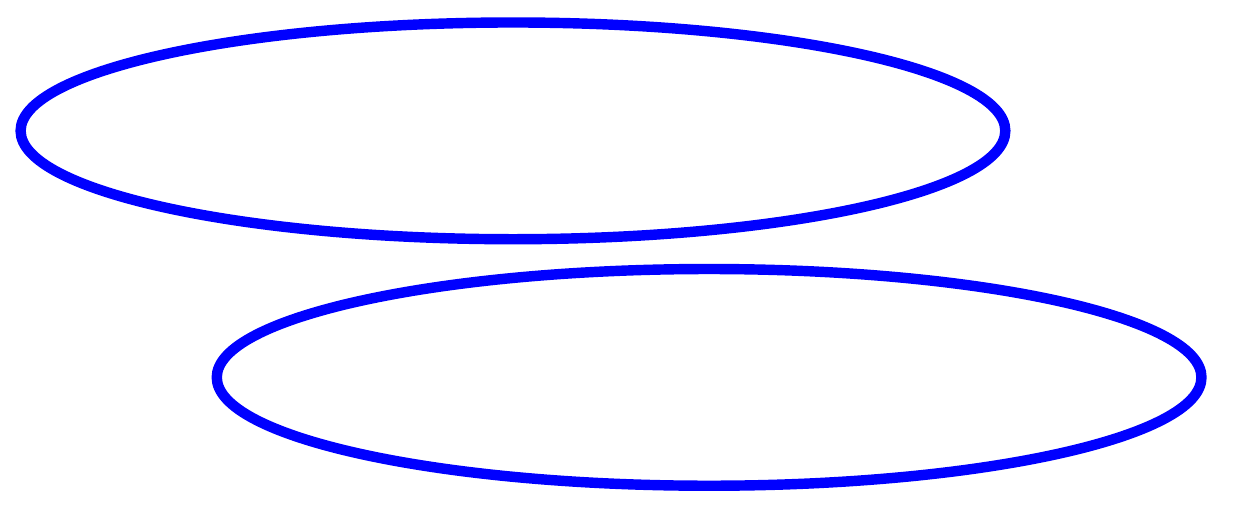}
        \caption{Case I: the boundary $\Gamma$ is a smooth curve or union of smooth curves, on which generalized impedance boundary conditions are imposed. In this case, we denote $\Gamma=\Gamma_*$.}
        \label{fig:casei}
    \end{subfigure}
    \begin{subfigure}[t]{0.8\textwidth}
        \centering
        \includegraphics[width=\textwidth]{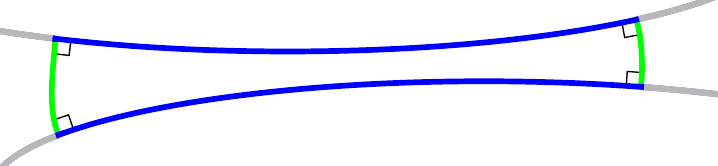}
        \caption{Case II: $\Gamma$ is the union of $\Gamma_*$ (blue) and a fictitious boundary $\Gamma_\o$ (green).}
        \label{fig:caseii}
    \end{subfigure}
    \caption{}
    \label{fig:cases}
\end{figure}

\subsection{Problem statement}

We consider boundary value problems of the form
\begin{equation}\label{eq:pde}\begin{cases}
    \Delta u + k^2 u = 0 & \Omega, \\
    c_1\lb u + c_2u +\partial_nu = f & \Gamma_*, \\
    au+\partial_nu = g & \Gamma_\circ , \\
    \partial_b u = h & \partial \Gamma_*.
\end{cases}\end{equation}
Here $a$ is a complex constant, $\lb$ denotes the Laplace-Beltrami operator, and $\partial_n$, $\partial_b$ denote the normal and binormal derivatives, respectively. 
The boundary value problem \eqref{eq:pde} is used in Case II, while in Case I, $\Gamma_\circ$ and $\del\Gamma_*$, and their associated boundary conditions, are omitted.
In Case I, if $\Omega$ is an exterior domain then an additional (outgoing) radiation condition
at infinity is imposed:

\begin{equation}
\label{eq:sommerfeld}
\lim_{|x|\to\infty} \frac{x}{|x|}\cdot\nabla u(x) - iku(x) = o\left(\frac{1}{\sqrt{|x|}}\right) \; .
\end{equation}

\subsection{Mathematical preliminaries}

The (outgoing) free space Green's function for the Helmholtz equation in two dimensions is given by
\begin{equation} G(x,y) = \frac{i}{4}H_0^{(1)}(k|x-y|), \end{equation}
where $H_0^{(1)}$ is the zeroth order Hankel function of the first kind (see \cite[\href{https://dlmf.nist.gov/10.2.E5}{(10.2.5)}]{NIST:DLMF}), and satisfies the PDE
\begin{equation}\label{eq:greens}
\Delta_x G(x,y) + k^2G(x,y) = -\delta(x-y),
\end{equation}
together with the radiation condition~\eqref{eq:sommerfeld}.

Given a piecewise smooth curve $\Gamma$, the \emph{single layer} and \emph{double layer} operators are defined, respectively, by
\begin{align}
    S[\sigma](x) &= \int_\Gamma G(x,y)\sigma(y)\,{\rm d}s(y),\label{eq:S} \\ 
    D[\sigma](x) &= \int_\Gamma \left[\partial_{n_y}G(x,y)\right]\sigma(y)\,{\rm d}s(y),\label{eq:D}
\end{align}
where $n_y$ denotes the outward normal vector at $y\in\Gamma$ and $\partial_n=n\cdot\nabla$ denotes the normal derivative. For $\sigma\in L^2(\Gamma)$, \eqref{eq:S} and \eqref{eq:D} automatically satisfy the Helmholtz equation in $\mathbb{R}^2\backslash \Gamma$. For the evaluation
of the single and double layer potentials on surface, i.e. for $x\in \Gamma$, we use the notations $\cS[\sigma]$ and 
$\cD[\sigma]$, respectively. We will also need the on-surface normal derivatives of $S$ and $D,$ defined as follows:
\begin{align}
    \cS'[\sigma](x) &= \int_\Gamma \left[\partial_{n_x}G(x,y)\right]\sigma(y)\,{\rm d}s(y), \\ 
    \cD'[\sigma](x) &= \int_\Gamma \left[\partial_{n_x}\partial_{n_y}G(x,y)\right]\sigma(y)\,{\rm d}s(y),
\end{align}
for all $x \in \Gamma.$ Note that $\partial_{n_x}\partial_{n_y}G$ is not integrable on $\Gamma$, so the integral defining the hypersingular operator $D'$ is interpreted in the Hadamard finite part sense \cite{colton_integral_2013}. We will also use the notation $\cS''$ and $\cD''$ for the
analogous operators with two normal derivatives on the $x$ variables, whose integrals must also be interpreted in
the finite part sense~\cite{kolm_quadruple_2003}.

For $x\in\Gamma$ not a corner point, the operators $D$ and $S'$ satisfy the following standard jump relations: 
\begin{align}
    \lim_{h\to 0} D[\sigma](x\pm hn_x) &= \pm\frac{\sigma(x)}{2}+\cD[\sigma](x),\label{eq:Djump} \\
    \lim_{h\to 0} n_x\cdot\nabla S[\sigma](x\pm hn_x) &= \mp\frac{\sigma(x)}{2}+\cS'[\sigma](x).\label{eq:S'jump}
\end{align}
The operators $S$ and $D'$ are continuous across $\Gamma$:
\begin{align}
    \lim_{h\to 0} S[\sigma](x\pm hn_x) &= \cS[\sigma](x),\label{eq:Scont} \\
    \lim_{h\to 0} n_x\cdot\nabla D[\sigma](x\pm hn_x) &= \cD'[\sigma](x).\label{eq:D'cont}
\end{align}

The operators $\cS,\cD,$ and $\cS'$ are bounded from $L^2(\Gamma)$ to itself, while $\cD'$ is bounded from $H^1(\Gamma)$ to $L^2(\Gamma)$. If $\Gamma$ is smooth, then $\cS,\cD,$ and $\cS'$ are compact on $L^2(\Gamma)$ \cite{colton_inverse_2013}. For domains with corners, though $\cS$ remains compact on $L^2(\Gamma)$, both $\cD$ and $\cS'$ are merely bounded in the vicinity of each corner \cite{fabes_spectra_1976}. 

\section{Integral equation formulation}

Before constructing our integral representation ansatz for problems with the generalized impedance
condition~\eqref{eq:vtbc}, we first derive several properties
of the Laplace-Beltrami operator applied to the standard Helmholtz layer 
potentials in~\cref{sec:lblayer}. Then, we describe a representation for~\eqref{eq:vtbc}
on a single curve in isolation that leverages cancellations between hypersingular
operators, along with an analytic preconditioning
strategy, in~\cref{sec:vtbcrep}. For a closed curve, as in Case I, these techniques are sufficient to
obtain a Fredholm second kind equation. For an open curve with boundary
conditions on the binormal data, as in Case II, the derivation is formal, as
modifications of the layer potentials are required to ensure that some of the
operators are well-defined. These issues are treated along with the full
mixed boundary value problem for Case II in \cref{sec:mixedbcrep}, which makes
use of an extension of the method of images. Finally, we discuss the use of 
impedance-to-impedance maps for domain decomposition in this context in 
\cref{subsec:dd}.

\subsection{The Laplace-Beltrami operator and Helmholtz layer potentials}
\label{sec:lblayer}
Consider a smooth, possibly open, planar curve $\Gamma$. For $h>0,$ let
\begin{equation} 
\Gamma_h = \{x-hn_x:x\in\Gamma\}.
\end{equation}
We further require $h$ to be sufficiently small so that $\Gamma_h$ does not self-intersect. We then define the operators
\begin{align}
    \cS_{\tosp{\Gamma}{\Gamma_h}}&:L^2(\Gamma)\to L^2(\Gamma_h), \\
    \cD_{\tosp{\Gamma}{\Gamma_h}}&:L^2(\Gamma)\to L^2(\Gamma_h),
\end{align}
by the restriction of the usual layer potentials for densities on $\Gamma$ to $\Gamma_h$.
The curvature on $\Gamma$ is defined by  
\begin{equation}
\kappa=\nabla_\Gamma\cdot n.
\end{equation}
The following lemma is a variation of a fact found, for example, in \cite{oneil_second-kind_2018}:
\begin{lemma}\label{lem:lbSlim} Suppose that $\Gamma$ is a smooth curve and $\sigma \in H^1(\Gamma)$.
Then
\begin{equation} \lim_{h\to0}\lb \cS_{\tosp{\Gamma}{\Gamma_h}}[\sigma] = \lb \cS [\sigma] = -k^2 \cS[\sigma] -\kappa \cS'[\sigma] - \cS''[\sigma]. \end{equation}
\end{lemma}
\begin{proof}
    % For $\sigma\in L^2(\Gamma)$, $S\sigma$ is analytic in $\mathbb{R}^2\backslash\Gamma$, and is continuous across $\Gamma$. Since 
    % \begin{equation} 
    % (\Delta+k^2)S\sigma=0\textrm{ in }\mathbb{R}^2\backslash\Gamma, 
    % \end{equation}
    % this implies that 
    % \begin{equation}\label{eq:Shelmsurf}
    %     (\Delta+k^2)S\sigma=0\textrm{ on }\Gamma.
    % \end{equation} 
    % The Laplacian is decomposed as 
    % \begin{equation}\label{eq:lapdecomp}
    %     \Delta = \lb+\kappa\partial_n+\partial_n^2.
    % \end{equation}
    % Inserting this into \eqref{eq:Shelmsurf} yields the claim.
    The Laplacian can be decomposed as\begin{equation}\label{eq:lapdecomp}
        \Delta = \lb+\kappa\partial_n+\partial_n^2.
    \end{equation}
    Proof of this fact can be found, for example, in \cite{nedelec_acoustic_2001}. Since $S[\sigma]$ satisfies the Helmholtz equation in $\mathbb{R}^2\backslash\Gamma$, 
    \begin{equation}
        \lb \cS_{\tosp{\Gamma}{\Gamma_h}}[\sigma] = -k^2 \cS_{\tosp{\Gamma}{\Gamma_h}}[\sigma] -\kappa_h \cS'_{\tosp{\Gamma}{\Gamma_h}}[\sigma] - \cS''_{\tosp{\Gamma}{\Gamma_h}} [\sigma] \; ,
    \end{equation}
    where $\kappa_h$ denotes the curvature on $\Gamma_h$. The operator $S$ is continuous across $\Gamma$ in the sense of \eqref{eq:Scont}. The jump of $\cS'$ is given in \eqref{eq:S'jump}. The following formula for the jump of $\cS''$ is derived in \cite{kolm_quadruple_2003}:
    \begin{equation}
        \lim_{h\to0}\cS''_{\tosp{\Gamma}{\Gamma_h}}[\sigma] = -\frac{\kappa}{2}\sigma +\cS''[\sigma].
    \end{equation}
    Thus the jump terms of $S'$ and $S''$ cancel, yielding the claim.
\end{proof}
The following lemma gives a similar result for the double layer potential.
\begin{lemma}\label{lem:lbDlim} Suppose that $\Gamma$ is a smooth curve and $\sigma \in H^2(\Gamma)$.
Then
\begin{equation} \lim_{h\to0} \lb \cD_{\tosp{\Gamma}{\Gamma_h}}[\sigma] = \lb \lim_{h\to0}\cD_{\tosp{\Gamma}{\Gamma_h}} [\sigma] = -\frac{1}{2}\lb \sigma +K [\sigma], \end{equation}
where the operator
\begin{equation}
    K = -k^2\cD-\kappa \cD'-\cD''
\end{equation}
has logarithmically singular kernel.
\end{lemma}
\begin{proof}
    The proof is given in Appendix~\ref{sec:appA}.
\end{proof}

\subsection{An integral representation on $\Gamma_*^j$}
\label{sec:vtbcrep}
Consider a single curve, $\Gamma_*^j$, on which we impose the generalized impedance condition~\cref{eq:vtbc}.
To obtain a Fredholm second kind formulation of this problem, we use a particular linear combination of 
$S$ and $D$ constructed to cancel singularities, reducing the problem to an integro-differential equation on the boundary $\Gamma_*^j$. The resulting integro-differential equation is of a convenient form that lends itself easily to preconditioning.

Substituting the representation $u=S[\sigma]$ into the boundary condition on $\Gamma_*^j$ and using Lemma \ref{lem:lbSlim} gives the expression
\begin{equation}\label{eq:srep} c_1\lb \cS[\sigma] + c_2\cS[\sigma] + \left(\frac{I}{2} + \cS'\right)[\sigma]. \end{equation}
The operator $\lb \cS$ is hypersingular. On the other hand, if we use the representation $u=D[\sigma]$ then Lemma \ref{lem:lbDlim} yields
\begin{equation}\label{eq:drep} c_1\lb\left(-\frac{I}{2}+\cD\right)[\sigma] + c_2\left(-\frac{I}{2}+\cD\right)[\sigma] + \cD'[\sigma]. \end{equation}
Though the operator $\cD'$ is also hypersingular, we note that the leading order singularities of $\cD'$ and $\cS''$ in the sum $\cD'+\cS''$ cancel, and what remains is a continuous kernel. Applying Lemma \ref{lem:lbSlim}, we find that
\begin{equation}
\cD'-\lb \cS = (\cD'+\cS'')+\kappa \cS'+k^2 \cS \; .
\end{equation}
Thus, taking the representation 
\begin{equation} u=\left(\cD-\frac{1}{c_1}\cS\right)[\sigma]\end{equation}
cancels the leading order singularities of $\lb \cS$ and $\cD'$ from \eqref{eq:srep} and \eqref{eq:drep}. Substituting this into the boundary conditions yields, after some manipulations, the integro-differential equation
\begin{align}\label{eq:int_diff} f =\Bigg[-\frac{c_1}{2}\bigg(\lb&+\frac{1+c_1c_2}{c_1^2}I\bigg)+(c_1\lb+c_2I)\cD+c_2\left(-\frac{I}{2}+\cD\right)\\
&\nonumber+(\cD'+\cS'')+k^2\cS+\kappa \cS'-\frac{c_2}{c_1}\cS-\frac{1}{c_1}\left(\frac{I}{2}+\cS'\right)\Bigg][\sigma]. \end{align}

Next, we define the \textit{surface wave number} on $\Gamma_*$ as
\begin{equation} \k = \sqrt{\frac{1+c_1c_2}{c_1^2}}, \end{equation}
with the branch of the square root taken such that $\Im \k > 0$ (strict positivity is
guaranteed by the form of $c_1$ and $c_2$) and we define a Helmholtz--Beltrami operator $\mathcal{L}$ on $\Gamma_*$ by
\begin{equation} \L = \lb+\k^2I. \end{equation}
We note that on a curve, the Laplace-Beltrami operator is the same as the second derivative with respect to arc length. In the following we let $\mathcal{L}_j$ denote the restriction of $\mathcal{L}$ to $\Gamma_*^j$. The surface Green's function associated with $\L_j$ is, therefore,
\begin{equation}\label{eq:surf_greens} \gj(s,s') = \frac{1}{2ik_\Gamma}\mathrm{e}^{ik_\Gamma|s-s'|}+\psi_j(s,s'), \end{equation}
where $s,s'$ denote arc length, and $\psi_j$ is smooth in $s$ and $s'$ separately and satisfies
\begin{equation}
    \L_j\psi_j(s,s_0) = \L_j\psi_j(s_0,s') = 0,
\end{equation}
for every $s_0\in[0,L]$, where $L$ is the length of $\Gamma_*^j$. Let
\begin{equation}
    \G_j[\sigma](s) = \int_{\Gamma_*^j} \gj(s,s')\sigma(s')\,{\rm d}s' \; .
\end{equation}
Standard ODE results give that $\G_j:L^2(\Gamma_*^j)\to H^2(\Gamma_*^j)$ is bounded.

\subsubsection{Closed $\Gamma_*^j$}

Suppose now that $\Gamma_*^j$ is a smooth, closed curve. In this case, $\gj$ should be the periodic Green's function, and $\psi_j$ in \eqref{eq:surf_greens} is given by
\begin{equation} 
    \psi_j(s,s') = \frac{1}{2i\k}\sum_{\ell\neq0} \ev^{i\k|s-s'-\ell L|}.  
\end{equation}
For $\Im\k>0$, as is the case for the physical values of $c_1$ and $c_2$ in \eqref{eq:c1} and \eqref{eq:c2}, the above series is absolutely convergent. For $\varphi \in H^2(\Gamma_*^j),$ integration by parts on $\Gamma_*^j$ yields
\begin{equation}\label{eq:ibpI}
    \G_j\L_j\varphi = \varphi.
\end{equation}
Noting that
\begin{equation}\label{eq:hbopid}
    c_1\lb+c_2I=c_1\L-c_1^{-1}I,
\end{equation} 
from \eqref{eq:int_diff} one obtains
\begin{align}\label{eq:lpieI}
    \G_j f = \Bigg[-\frac{c_1}{2}I+c_1 \cD&-\frac{1}{c_1}\G_j \cD+\G_j(\cD'+\cS'')+\left(k^2-\frac{c_2}{c_1}\right)\G_j \cS\\
    &\hspace{120pt}+\G_j\left(\kappa-\frac{1}{c_1}\right) \cS'\Bigg]\sigma.\nonumber
\end{align}
The integral operator in the above equation is Fredholm second-kind on $H^2(\Gamma_*^j)$.

\subsubsection{Open $\Gamma_*^j$}
\label{sec:vtbcrepopen}
To obtain an integral equation on an open $\Gamma_*^j$, we proceed as in the previous section, except that further work is required to address the additional boundary conditions.
We define
\begin{align}
    F_j[\varphi](s) = \gj(s,L)\varphi'(L) &- \gj(s,0)\varphi'(0).
\end{align}
\begin{lemma}\label{lem:Fj_cmpct}
    The operator $F_j$ is compact on $H^2(\Gamma_*^j)$.
\end{lemma}
\begin{proof}
    Since the embedding $H^2(\Gamma_*^j)\subset C^1(\Gamma_*^j)$ is continuous \cite{aubin_nonlinear_1982}, the functionals 
    \begin{align}
        \varphi&\mapsto \varphi'(0),\\
        \varphi&\mapsto\varphi'(L)
    \end{align} 
    are bounded on $H^2(\Gamma_*^j)$. The range of $F_j$ is a two-dimensional subspace of $C^\infty(\Gamma_*^j)$, and in particular of $H^2(\Gamma_*^j)$. Hence $F_j$ is a bounded operator on $H^2(\Gamma_*^j)$ with finite-dimensional range, so it is compact.
\end{proof}
We select $\psi_j$ in \eqref{eq:surf_greens} such that $\gj$ is the Green's function with zero Neumann boundary conditions, i.e. $\partial_{s'}\gj(s,L)$ and $\partial_{s'}\gj(s,0)$ are identically zero. 
Integration by parts on $\Gamma_*^j$ then yields, for $\varphi\in H^2(\Gamma_*^j)$, 
\begin{equation}\label{eq:ibpII}
    \G_j\L_j\varphi = (I+F_j)\varphi.
\end{equation}
Let $f_j$ and $\sigma_j$ denote the restrictions of $f$ and $\sigma$ to $\Gamma_*^j$. Using \eqref{eq:hbopid}, from \eqref{eq:int_diff} one obtains a non-periodic analogue of \eqref{eq:lpieI}:
\begin{align}\label{eq:lpieII}
    \G_j f_j = \Bigg[-\frac{c_1}{2}I+c_1\cD&-\frac{1}{c_1}\G_j \cD+\G_j(\cD'+\cS'')+\left(k^2-\frac{c_2}{c_1}\right)\G_j \cS\\
    &\hspace{20pt}+\G_j\left(\kappa-\frac{1}{c_1}\right) \cS' + c_1F_j\left(-\frac{I}{2}+\cD\right)\Bigg]\sigma.\nonumber
\end{align}
The operator 
\begin{equation}
    \mathcal{R} = \frac{1}{c_1}\cD+(\cD'+\cS'')+\left(k^2-\frac{c_2}{c_1}\right) \cS+\left(\kappa-\frac{1}{c_1}\right) \cS'
\end{equation}
has logarithmically singular kernel, and is thus compact on $L^2(\Gamma_*^j)$. The embedding 
\begin{equation}
    H^2(\Gamma_*^j)\overset{\iota}{\hookrightarrow} L^2(\Gamma_*^j)
\end{equation}
is compact \cite{aubin_nonlinear_1982,aubin_espaces_1976}, so the composition $\G_j\,\o\,\mathcal{R}\circ\iota$ is compact on $H^2(\Gamma_*^j)$. Since the most singular part of the kernel of $\cD$ is $\mathcal{O}(r^2\log r)$, $\cD$ is compact on $H^2(\Gamma_*^j)$. The operator $F_j$ is compact by Lemma~\ref{lem:Fj_cmpct}, so this proves that the operator applied to $\sigma$ in \eqref{eq:lpieII} is Fredholm second kind on
$H^2(\Gamma_*^j)$. 

The operator has non-trivial cokernel because its range is contained in the range of 
$\G_j$, which only contains functions with zero Neumann data. This deficiency is related to the
binormal boundary conditions, which must still be imposed. 

For this representation, the binormal boundary conditions take the form
\begin{align} 
-\frac{1}{2}\sigma'(L)+\partial_s(\cD\sigma)(L)-\frac{1}{c_1}\partial_s(\cS\sigma)(L) &= h_+,\label{eq:bnbc1} \\
-\frac{1}{2}\sigma'(0)+\partial_s(\cD\sigma)(0)-\frac{1}{c_1}\partial_s(\cS\sigma)(0) &= -h_-.\label{eq:bnbc2}
\end{align}
Unfortunately, the value of $\partial_s(\cS\sigma)$ is not necessarily defined at the end
points. We will see how to modify the definition of $S$ to alleviate this issue in
the next subsection. For now, we proceed formally. 

Multiplying \eqref{eq:bnbc1} by $\gj(s,L)$, multiplying \eqref{eq:bnbc2} by $\gj(s,0)$, and taking the difference of the two yields
\begin{equation}\label{eq:bnbc_diff} 
F_j\left(-\frac{I}{2}+\cD\right)\sigma = \frac{1}{c_1}F_j\cS\sigma + \gj(s,L)h_+ + \gj(s,0)h_-.
\end{equation} 
Since $\gj(s,L)$ and $\gj(s,0)$ are linearly independent, imposing this is equivalent to imposing \eqref{eq:bnbc1} and \eqref{eq:bnbc2}. 
Substituting \eqref{eq:bnbc_diff} back into \eqref{eq:lpieII} yields the following (formal) integral equation:
\begin{multline}\label{eq:lpieII_mod}
    \tilde{f}_j = \left ( -\frac{c_1}{2}I+c_1\cD-\frac{1}{c_1}\G_j \cD+\G_j(\cD'+\cS'') \right.\\
    \left. +\left(k^2-\frac{c_2}{c_1}\right)\G_j \cS +\G_j\left(\kappa-\frac{1}{c_1}\right) \cS' + F_j\cS \right )\sigma \; ,
\end{multline}
where 
\begin{equation} \label{eq:modf} \tilde{f}_j = \G_j f_j - \gj(s,L)h_+ - \gj(s,0)h_-. \end{equation}
As mentioned above, when enforcing the binormal boundary condition the operator $\cS$ does not map $H^2(\Gamma_*^j)$ to itself when $\Gamma_*^j$ is an open curve. This issue is handled by 
an extension of the method of images described in the next subsection.

\subsection{Mixed boundary conditions and a method of images approach}
\label{sec:mixedbcrep}

Before describing the image methods, we consider a provisional representation for the mixed-boundary problem on all of $\Gamma$. We introduce the notation $S_{*_i}$ for the layer potential $S$ with density on $\Gamma_*^i$ and $S_{\o_i}$ for the layer potential $S$ with density on $\Gamma_\o^i$, and will use an analogous notation for $D$. We also let $\cS_{\tosp{*_i}{\o_j}}$ denote the restriction of the operator $S_{*_i}$ to $\Gamma_\o^j$, and similarly for the other boundary components and operators, with $\cS_{*_i}:=\cS_{\tosp{*_i}{*_i}}$, etc. used for the sake of brevity. For ease of exposition we consider the problem depicted in Figure \ref{fig:caseii} in which there are two pieces of the boundary with the generalized impedance boundary conditions ($\Gamma_*^{1},\Gamma_*^{2}$) and two pieces with Robin boundary conditions ($\Gamma_\o^{1},\Gamma_\o^{2}$). Then, following the
discussion above, we provisionally represent the solution as
\begin{equation} u=\left(D_{*_1}-\frac{1}{c_1}S_{*_1}\right)\sigma_1+\left(D_{*_2}-\frac{1}{c_1}S_{*_2}\right)\sigma_2+S_{\o_1}\rho_1+S_{\o_2}\rho_2. \end{equation}
Upon substituting this representation into the boundary conditions, one obtains the $4\times4$ integral equation system 
\begin{equation}\label{eq:fullsys_nofin} \begin{bmatrix} \K_{\tosp{*_1}{*_1}} & \K_{\tosp{\o_1}{*_1}} & \K_{\tosp{*_2}{*_1}} & \K_{\tosp{\o_2}{*_1}} \\ 
\K_{\tosp{*_1}{\o_1}} & \K_{\tosp{\o_1}{\o_1}} & \K_{\tosp{*_2}{\o_1}} & \K_{\tosp{\o_2}{\o_1}} \\ 
\K_{\tosp{*_1}{*_2}} & \K_{\tosp{\o_1}{*_2}} & \K_{\tosp{*_2}{*_2}} & \K_{\tosp{\o_2}{*_2}} \\ \K_{\tosp{*_1}{\o_2}} & \K_{\tosp{\o_1}{\o_2}} & \K_{\tosp{*_2}{\o_2}} & \K_{\tosp{\o_2}{\o_2}} \end{bmatrix} \begin{bmatrix} \sigma_1 \\ \rho_1 \\ \sigma_2 \\ \rho_2 \end{bmatrix} = \begin{bmatrix} \tilde{f}_1 \\ g_1 \\ \tilde{f}_2 \\ g_2 \end{bmatrix} \; , \end{equation}
where the modified data $\tilde{f}_1,\tilde{f}_2$ are defined as in~\cref{eq:modf}. We now account for the singularities occurring in each block of the integral operator system above. In Section~\ref{sec:vtbcrepopen}, we found that $\K_{\tosp{*_i}{*_i}}$ is second-kind on $H^2(\Gamma_*^i)$ up to terms involving $\cS_{*_i}$. We also have that
\begin{equation} \K_{\tosp{\o_j}{\o_j}}=\frac{1}{2}I + a\cS_{\o_j}+\cS_{\o_j}' \end{equation}
is second-kind on $L^2(\Gamma_\o^j)$. It remains to check the singularity of the integral operators acting between $\Gamma_*$ and $\Gamma_\o$. We have
\begin{align} 
\K_{\tosp{*_i}{\o_j}}&= a\left(\cD_{\tosp{*_i}{\o_j}}-\frac{1}{c_1}\cS_{\tosp{*_i}{\o_j}}\right)+\cD'_{\tosp{*_i}{\o_j}}-\frac{1}{c_1}\cS'_{\tosp{*_i}{\o_j}},\label{eq:st_to_o_nofin} \\
\K_{\tosp{\o_j}{*_i}} &= c_1(I+F_i)\cS_{\tosp{\o_j}{*_i}}-\frac{1}{c_1}\mathcal{G}_i\cS_{\tosp{\o_j}{*_i}}+\mathcal{G}_i\cS'_{\tosp{\o_j}{*_i}}.\label{eq:o_to_st_nofin}
\end{align}
The most singular part of \eqref{eq:st_to_o_nofin} is $\cD'_{\tosp{*_i}{\o_j}}$ and the most singular part of \eqref{eq:o_to_st_nofin} is $F_i\cS_{\tosp{\o_j}{*_i}}$.

To eliminate these singularities, we proceed as follows. First, we treat the singularity of $\cD'_{\tosp{*_i}{\o_j}}$ at the corner. For each boundary point $x_0\in\del\Gamma_*^i$, we consider the line $\ell$ passing through $x_0$ in the normal direction. For some $R>0$, we then define a curve via reflection of $\Gamma_*^i\cap B_R(x_0)$ across $\ell$, which we call a \emph{fin}.
Next we define a modified double layer potential, $\tilde{D}_{*_i}\sigma_i$, which is obtained by 
extending the values of $\sigma_i$ to the fins, with the values on the fin at each $x_0\in\del\Gamma_*^i$ 
defined by an even symmetric extension of the values of $\sigma_i$ on $\Gamma_*^i\cap B_R(x_0)$,
and evaluating the usual double layer induced by the extended density. An easy consequence of the even symmetry of the extended density, together with the symmetry of the extended domain and the additional assumptions on $\Gamma_\o$ (see \cref{subsec:geomdesc}), is that the kernel of the operator $\tilde{\cD}'_{\tosp{*_i}{\o_j}}$ is continuous away from the corner and bounded for source and target approaching the corner, and so the operator is compact (see Appendix\ref{sec:appB}). 
We define $\tilde{\cS}_{*_i}$ by even extension as well, which has the benefit that 
$\partial_s (\tilde{\cS}_{*_i}\sigma_i)$ is defined at the endpoints.

Similarly, the singularity in $F_i\cS_{\tosp{\o_j}{*_i}}$ can be eliminated by extending $\Gamma_\o^j$ by a fin and using an {\em odd} extension
of $\rho_j$ to define $\tilde{S}_{\o_j}\rho_j$. 
A proof of this fact is given in Appendix~\ref{sec:appB}. See \cref{fig:fins} for a zoom-in of the fins near a corner. 

\begin{remark}
    While the image fins above are defined by reflection across the line parallel to
    the normal at the endpoint, it is useful in other settings to define them
    by other reflections or rotations. 
    Similar fins were recently applied to regularize corner singularities that arose in 
    a domain decomposition approach to waveguide problems~\cite{goodwill2025domain}. The 
    regularizing effect of including image curves with certain symmetries has also 
    been observed in integral formulations of quasiperiodic scattering problems
    \cite{barnett2011new}.
\end{remark}

\begin{figure}[htbp]
    \centering
    \includegraphics[width=0.5\textwidth]{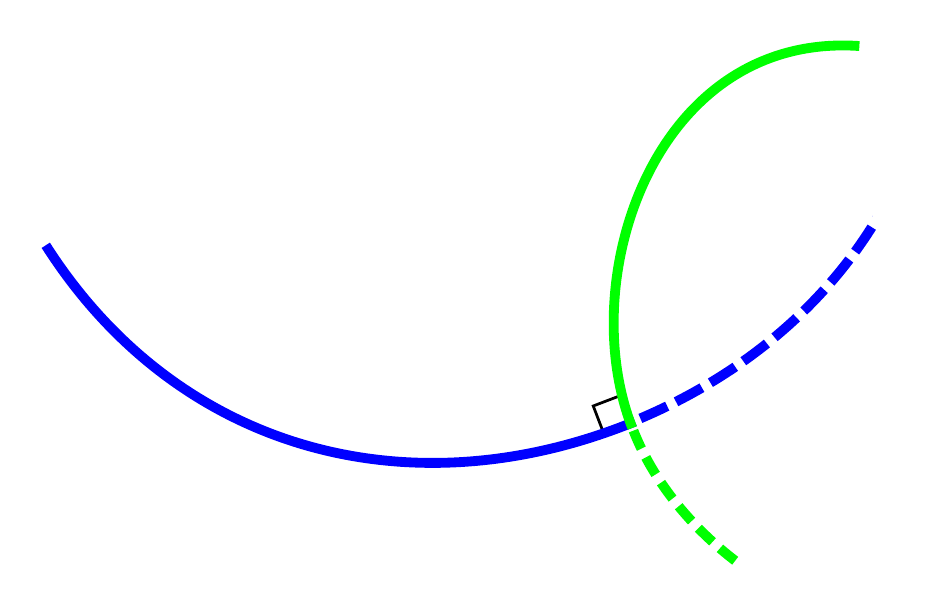}
    \caption{A zoom-in of the fins added to the geometry in the vicinity of a corner.}
    \label{fig:fins}
\end{figure}

We now propose the fin-based representation
\begin{equation} u=\left(\tilde{D}_{*_1}-\frac{1}{c_1}\tilde{S}_{*_1}\right)\sigma_1+\left(\tilde{D}_{*_2}-\frac{1}{c_1}\tilde{S}_{*_2}\right)\sigma_2+\tilde{S}_{\o_1}\rho_1+\tilde{S}_{\o_2}\rho_2 \; , \end{equation}
which uses even extensions for the $\Gamma_*^i$ operators and odd extensions
for the $\Gamma_{\o}^i$ operators. We make a substitution as in \eqref{eq:bnbc_diff} and \eqref{eq:lpieII_mod} so that the operator $F_i$ is only applied to $\tilde{S}_{*_i}$, outside of the enforcement of the binormal boundary condition. For $\star,\bullet\in\{*,\o\}$, let $\tK_{\tosp{\bullet_i}{\star_j}}$ denote the operator in \eqref{eq:fullsys_nofin} with the appropriate fins added. 

Explicitly, we have
\begin{align}
    \tK_{\tosp{*_i}{*_i}}\sigma_i &=  -\frac{c_1}{2}I_{*_i}+c_1\tilde{\cD}_{*_i}-\frac{1}{c_1}\G_i \tilde{\cD}_{*_i}+\G_i(\tilde{\cD}'_{*_i}+\tilde{\cS}''_{*_i})  \\
    &\qquad\qquad+\left(k^2-\frac{c_2}{c_1}\right)\G_i \tilde{\cS}_{*_i}+\G_i\left(\kappa-\frac{1}{c_1}\right) \tilde{\cS}'_{*_i} + F_i\tilde{\cS}_{*_i}, \\
    \tK_{\tosp{*_i}{\o_j}} &= a\left(\tilde{\cD}_{\tosp{*_i}{\o_j}}-\frac{1}{c_1}\tilde{\cS}_{\tosp{*_i}{ \o_j}}\right)+\tilde{\cD}'_{\tosp{*_i}{\o_j}}-\frac{1}{c_1}\tilde{\cS}'_{\tosp{*_i}{\o_j}}, \\
    \tK_{\tosp{*_i}{*_{\neg i}}} &= c_1\left(\tilde{\cD}_{\tosp{*_i}{*_{\neg i}}}-\frac{1}{c_1}\tilde{\cS}_{\tosp{*_i}{*_{\neg i}}}\right)-\frac{1}{c_1}\G_{\neg i}\left(\tilde{\cD}_{\tosp{*_i}{*_{\neg i}}}-\frac{1}{c_1}\tilde{\cS}_{\tosp{*_i}{*_{\neg i}}}\right)\\
    &\hspace{50pt}+\G_{\neg i}\left(\tilde{\cD}'_{\tosp{*_i}{*_{\neg i}}}-\frac{1}{c_1}\tilde{\cS}'_{\tosp{*_i}{*_{\neg i}}}\right),\nonumber \\
    \tK_{\tosp{\o_j}{\o_j}} &= \frac{1}{2}I_{\o_j}+a\tilde{\cS}_{\o_j}+\tilde{\cS}_{\o_j}', \\
    \tK_{\tosp{\o_j}{*_i}} &= c_1\tilde{\cS}_{\tosp{\o_j}{*_i}}-\frac{1}{c_1}\mathcal{G}_i\tilde{\cS}_{\tosp{\o_j}{*_i}}+\mathcal{G}_i\tilde{\cS}'_{\tosp{\o_j}{*_i}}, \\
    \tK_{\tosp{\o_j}{\o_{\neg j}}} &= a\tilde{\cS}_{\tosp{\o_j}{\o_{\neg j}}}+\tilde{\cS}'_{\tosp{\o_j}{\o_{\neg j}}}.
\end{align}

Then, the full system is given by
\begin{equation}\label{eq:fullsys} \begin{bmatrix} \tK_{\tosp{*_1}{*_1}} & \tK_{\tosp{\o_1}{*_1}} & \tK_{\tosp{*_2}{*_1}} & \tK_{\tosp{\o_2}{*_1}} \\ 
\tK_{\tosp{*_1}{\o_1}} & \tK_{\tosp{\o_1}{\o_1}} & \tK_{\tosp{*_2}{\o_1}} & \tK_{\tosp{\o_2}{\o_1}} \\ 
\tK_{\tosp{*_1}{*_2}} & \tK_{\tosp{\o_1}{*_2}} & \tK_{\tosp{*_2}{*_2}} & \tK_{\tosp{\o_2}{*_2}} \\ \tK_{\tosp{*_1}{\o_2}} & \tK_{\tosp{\o_1}{\o_2}} & \tK_{\tosp{*_2}{\o_2}} & \tK_{\tosp{\o_2}{\o_2}} \end{bmatrix} \begin{bmatrix} \sigma_1 \\ \rho_1 \\ \sigma_2 \\ \rho_2 \end{bmatrix} = \begin{bmatrix} \tilde{f}_1 \\ g_1 \\ \tilde{f}_2 \\ g_2 \end{bmatrix} \; , \end{equation}
Denote the operator in \eqref{eq:fullsys} by $\mathcal{M}$. The following theorem states that, up to a bounded operator with bounded inverse, \eqref{eq:fullsys} is a second-kind integral equation.
\begin{theorem}\label{thm:fullsysfred}
   The operator $\mathcal{M}:H^2(\Gamma_*)\times L^2(\Gamma_\o)\to H^2(\Gamma_*)\times L^2(\Gamma_\o)$ is of the form $A+K$, where $K$ is compact and $A$ is boundedly invertible.
\end{theorem}
\begin{proof}
    The operators $\frac{c_1}{2}I+\tK_{\tosp{*_i}{*_i}}$, $-\frac{1}{2}I+\tK_{\tosp{\o_j}{\o_j}}$, $\tK_{\tosp{*_i}{\neg*_i}}$, and $\tK_{\tosp{\o_j}{\neg\o_j}}$ are compact, on their respective spaces. The operators 
    \begin{align} 
    \tilde{\cS}'_{\tosp{\o_j}{*_i}}&:L^2(\Gamma_\o^j)\to L^2(\Gamma_*^i),\label{eq:S'corn} \\
    \tilde{\cD}_{\tosp{*_i}{\o_j}}&:L^2(\Gamma_*^i)\to L^2(\Gamma_\o^j)\label{eq:Dcorn} \end{align}
    are bounded \cite{fabes_spectra_1976}. The operator 
    \begin{equation} \G_i:L^2(\Gamma_*^i)\to H^2(\Gamma_*^i) \end{equation} 
    is also bounded, so the composition 
    \begin{equation} \G_i\tilde{\cS}'_{\tosp{\o_j}{*_i}}:L^2(\Gamma_\circ^j)\to H^2(\Gamma_*^i) \end{equation} 
    is bounded. The odd-symmetric fins on $\Gamma_\circ^j$ render 
    \begin{equation} \tilde{\cS}_{\tosp{\o_j}{*_i}}:L^2(\Gamma_\circ^j)\to H^2(\Gamma_*^i) \end{equation} 
    compact, and thus $\tK_{\tosp{\o_j}{*_i}}$ is bounded. Since the inclusion 
    \begin{equation} H^2(\Gamma_*^i)\hookrightarrow L^2(\Gamma_*^i) \end{equation} 
    is compact (\cite{aubin_nonlinear_1982,aubin_espaces_1976}) and \eqref{eq:Dcorn} is bounded, the composition \begin{equation} \tilde{\cD}_{\tosp{*_i}{\o_j}}:H^2(\Gamma_*^i)\to L^2(\Gamma_\o^j) \end{equation}
    is compact, and thus so too is $\tK_{\tosp{*_i}{\o_j}}$.

    Collect the compact operators in a 4x4 matrix of operators, denoted $K$, and let $A=2(\mathcal{M}-K)$. Then 
    \begin{equation} A = \begin{bmatrix} -c_1 I_{\tosp{*_1}{*_1}} & A_{11} && A_{12} \\
    & I_{\tosp{\o_1}{\o_1}} && \\
    & A_{21} & -c_1I_{\tosp{*_2}{*_2}} &  A_{22} \\
    &&&  I_{\tosp{\o_2}{\o_2}}\end{bmatrix},\end{equation}
    where the operators $A_{ij}:L^2(\Gamma_\o^j)\to H^2(\Gamma_*^i)$ are bounded. We then have
    \begin{equation}
        \Pi^{\tp} A\Pi = \begin{bmatrix} -c_1 I_{\tosp{*_1}{*_1}} && A_{11} & A_{12} \\
    & -c_1I_{\tosp{*_2}{*_2}} &  A_{21} &  A_{22} \\
    && I_{\tosp{\o_1}{\o_1}} &  \\
    &&&  I_{\tosp{\o_2}{\o_2}}\end{bmatrix},
    \end{equation}
    where $\Pi$ is a block permutation. The inverse of this operator is obtained by negating the strictly upper triangular part and multiplying by the appropriate diagonal matrix.
    % Write
    % \begin{align} 
    % 2A &= \begin{bmatrix} I_{\tosp{*_1}{*_1}} &&& A_{12}  \\
    % & I_{\tosp{\o_1}{\o_1}} && \\
    % && I_{\tosp{*_2}{*_2}} & \\
    % &&&  I_{\tosp{\o_2}{\o_2}}\end{bmatrix}\\
    % &\hspace{50pt}\times\begin{bmatrix} -c_1 I_{\tosp{*_1}{*_1}} & A_{11} && \\
    % & I_{\tosp{\o_1}{\o_1}} && \\
    % & A_{21} & -c_1I_{\tosp{*_2}{*_2}} &  A_{22} \\
    % &&&  I_{\tosp{\o_2}{\o_2}}\end{bmatrix}.\nonumber
    % \end{align}
    % The former matrix of operators is inverted by negating the off-diagonal operator. The latter has inverse
    % \begin{equation} \begin{bmatrix} -c_1^{-1} I_{\tosp{*_1}{*_1}} & c_1^{-1}A_{11} && \\
    % & I_{\tosp{\o_1}{\o_1}} && \\
    % & c_1^{-1}A_{21} & -c_1^{-1}I_{\tosp{*_2}{*_2}} &  c_1^{-1}A_{22} \\
    % &&&  I_{\tosp{\o_2}{\o_2}}\end{bmatrix}. \end{equation}
\end{proof}
At this point it is worth noting the following special case. Suppose that $f$ and $h$ are identically zero, as is the case in many problems of interest. In this case, using Schur complements, we reduce \eqref{eq:fullsys} to a second-kind integral equation solely on $\Gamma_\o$. Let
\begin{align}
    \tK_{\tosp{\o}{*_i}} &= \begin{bmatrix} \tK_{\tosp{\o_1}{*_i}} & \tK_{\tosp{\o_2}{*_i}} \end{bmatrix}, \\
    \tK_{\tosp{*}{\o_i}} &= \begin{bmatrix} \tK_{\tosp{*_1}{\o_i}} & \tK_{\tosp{*_2}{\o_i}} \end{bmatrix}, \\
    \tK_{\tosp{*_i}{\o}} &= \begin{bmatrix} \tK_{\tosp{*_i}{\o_1}} \\ \tK_{\tosp{*_i}{\o_2}} \end{bmatrix}, \\
    \tK_{\tosp{\o_i}{*}} &= \begin{bmatrix} \tK_{\tosp{\o_i}{*_1}} \\ \tK_{\tosp{\o_i}{*_2}} \end{bmatrix}, \\
    \tK_{\o} &= \begin{bmatrix}
\tK_{\tosp{\o_1}{\o_1}} & \tK_{\tosp{\o_2}{\o_1}} \\ 
\tK_{\tosp{\o_1}{\o_2}}& \tK_{\tosp{\o_2}{\o_2}} \end{bmatrix}.
\end{align}
Assuming that $\tK_{\tosp{*_i}{*_i}}$ is invertible, we define
\begin{equation}
    \mathcal{A}_i = \tK_{\tosp{*_i}{*_{\neg i}}}\tK_{\tosp{*_i}{*_i}}^{-1}, \\
\end{equation}
Since $\tK_{\tosp{*_i}{*_{\neg i}}}$ is compact, $\tK_{\tosp{*_i}{*_i}}-\mathcal{A}_{\neg i}\tK_{\tosp{*_i}{*_{\neg i}}}$ is also Fredholm second-kind. Suppose that it is invertible as well. We then define the operator
\begin{equation}
    \mathcal{B}_i = (\tK_{\tosp{*_i}{*_i}}-\mathcal{A}_{\neg i}\tK_{\tosp{*_i}{*_{\neg i}}})^{-1}(\mathcal{A}_{\neg i}\tK_{\tosp{\o}{*_{\neg i}}}-\tK_{\tosp{\o}{*_i}}).
\end{equation}
By taking Schur complements, \eqref{eq:fullsys} is equivalent to solving the reduced system
\begin{equation}\label{eq:reducsys}
    (\tK_{\o}+\tK_{\tosp{*_i}{\o}}\mathcal{B}_1+\tK_{\tosp{*_2}{\o}}\mathcal{B}_2)\rho = g.
\end{equation}
Note that $\mathcal{B}_i:\rho\mapsto\sigma_i$. Let the operator in \eqref{eq:reducsys} be denoted by $\mathcal{M}_{\o,r}$.
\begin{corollary}\label{cor:reducsysfred}
    $\mathcal{M}_{\o,r}:L^2(\Gamma_\o)\to L^2(\Gamma_\o)$ is Fredholm second-kind.
\end{corollary}
\begin{proof}
    In the proof of Theorem~\ref{thm:fullsysfred}, it was already shown that $\tK_{\tosp{\o}{*_i}}$ is bounded and $\tK_{\tosp{*_i}{\o}}$ is compact. Hence $\mathcal{B}_i$ is bounded, and the composition $\tK_{\tosp{*_i}{\o}}\mathcal{B}_i$ is compact. Since $\tK_\o$ is Fredholm second-kind, the statement follows.
\end{proof}

\subsection{Domain decomposition and impedance-to-impedance}\label{subsec:dd}

For complicated domains with multiscale features or modular components, it is often advantageous to divide the domain into multiple regions. Given a solver for each individual region, there are several methods of coupling them together, thus obtaining a solution to the full problem. Here, we briefly outline the use of impedance-to-impedance maps (see, e.g., ~\cite{gillman2015spectrally}) to enforce continuity of the solution and its normal derivative across each artificial boundary.

In the previous analysis, the solution $u$ for Case II is assumed to satisfy a general Robin boundary condition on the portion of the boundary denoted by $\Gamma_\o$, see \eqref{eq:pde}. As special cases, the solution is said to satisfy an incoming impedance condition if
\begin{equation}\label{eq:incimp} iku+\del_nu=g, \end{equation}
and an outgoing impedance condition if
\begin{equation}\label{eq:outimp} -iku+\del_nu=g. \end{equation}
Let $\mathcal{M}^+$ (resp. $\mathcal{M}^-$) denote the operator in \eqref{eq:fullsys} with $a=ik$ (resp. $a=-ik$). Similarly define $\mathcal{M}_{\o,r}^\pm$ for the reduced version of $\mathcal{M}^\pm$, given in \eqref{eq:reducsys}. Then the impedance-to-impedance maps are given by 
\begin{align}
    \mathcal{I}_{\tosp{+}{-}} &= \mathcal{M}_{\o,r}^-(\mathcal{M}_{\o,r}^+)^{-1}, \\
    \mathcal{I}_{\tosp{-}{+}} &= \mathcal{M}_{\o,r}^+(\mathcal{M}_{\o,r}^-)^{-1},
\end{align} 
if each $\mathcal{M}_{\o,r}^\pm$ is invertible. Hence the impedance-to-impedance maps are second-kind when they are defined. Once one has access to the impedance-to-impedance maps, by matching incoming and outgoing impedance data on each side of $\Gamma_\o$, continuity of the solution and its normal derivative is enforced. 

\begin{remark}
    Although the impedance-to-impedance maps are Fredholm second-kind, the system obtained by imposing continuity of the incoming and outgoing impedance data has a degeneracy in the identity part, making it only compact. However, empirically, we observe that the condition number of the system does not exceed $10^4$ for our discretizations.
\end{remark}

\section{Numerics}

For the numerical examples, we choose the parameters $c_1$ and $c_2$ based on
the acoustic boundary layer model in~\cite{berggren_acoustic_2018}, in which $u$ is 
a pressure field.
We denote the viscous and thermal acoustic boundary layer thicknesses by $\delta_V$ and $\delta_T$, respectively. These quantities depend in particular on the frequency $\omega$, and typically lie in the range of $20-400\,\mu\mathrm{m}$ for audible sound \cite{berggren_acoustic_2018}. The wave number is $k=\omega/c$, where $c$ is the speed of sound. The constant $\gamma \geq 1$ will denote the ratio of the specific heat capacities at constant pressure and constant volume.
The model of \cite{berggren_acoustic_2018} then takes the form \cref{eq:vtbc} where 
\begin{align}
    c_1 &= -\delta_V\frac{i-1}{2},\label{eq:c1} \\
    c_2 &= \delta_Tk^2(\gamma-1)\frac{i-1}{2}.\label{eq:c2}
\end{align}
In the figures below, as representative values we take boundary layer thicknesses $\delta_V=\delta_T=\frac{1}{160}$, specific heat ratio $\gamma=1.4$, and wavelength $\lambda=1.1$. 

For our simulations we use the chunkIE software package \cite{askham_chunkie_2024}. The package discretizes the boundary of a domain via curved panels of Gauss-Legendre nodes.
We use 16th order panels in the calculations below, and the panel lengths are selected to be less than $\lambda/4$. Additionally, for geometries with corners, we 
use dyadically refined panels near the corner points, to a depth of seven levels. For well-separated panels, the classical Gauss-Legendre quadrature is used. For adjacent panels and for the interaction of a panel with itself, adaptive integration and generalized Gauss-Legendre quadratures are used, see \cite{ma_generalized_1996,bremer2010nonlinear}. The package also implements a surface smoother (\cite{vico_fast_2020}) which we use for the generation of our example for Case I.

The geometry we use to illustrate Case I consists of a waveguide spiraling outward from a circular cavity. We consider the solution due to a point source located at $(x,y)=(-0.4,0.3)$. The absolute value of the solution is shown in figure \ref{fig:wg_csi_abs}.

\begin{figure}[htbp]
    \centering
    \includegraphics[width=\textwidth]{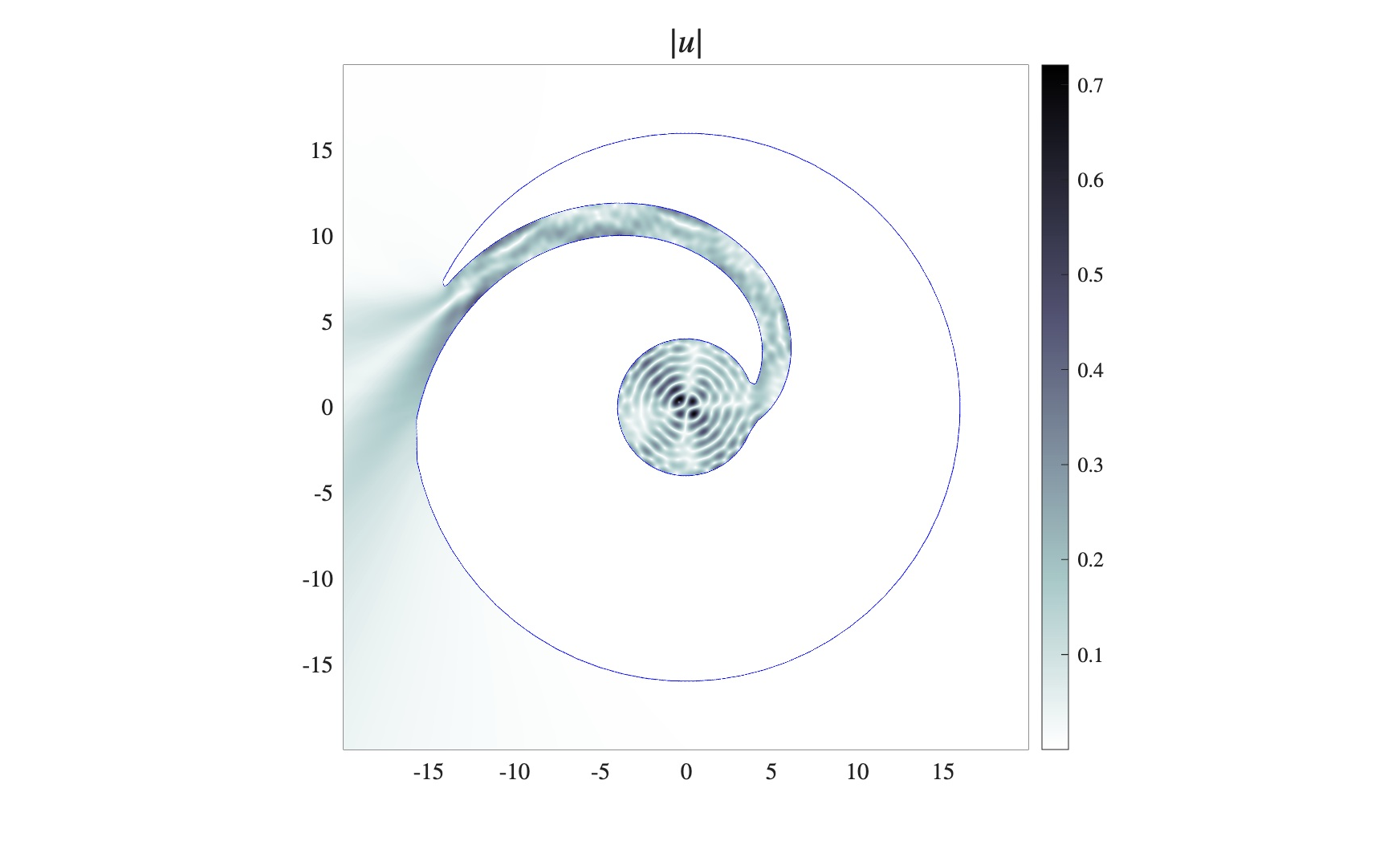}
    \caption{Magnitude of the solution due to a point source at $(x,y)=(-0.4,0.3)$, for the geometry outlined in blue.}
    \label{fig:wg_csi_abs}
\end{figure}

Our first example for Case II is a solution to \eqref{eq:pde}, where the domain $\Omega$ is taken to be a curved waveguide of finite length. The geometry is shown in figure \ref{fig:wg_geom}. On the boundary $\Gamma_*$ we take $f=0$, and on $\partial \Gamma_*$, $h=0$. Let $\Gamma_\o^1$ denote the left boundary, and let $\Gamma_\o^2$ denote the right boundary. Take $g\equiv 0$ on $\Gamma_\o^2$, and on $\Gamma_\o^1$ take $g$ to be the incoming impedance (see \eqref{eq:incimp}) due to a point source at $(x,y)=(-16.1,0)$. This is a rough approximation to the solution due to a point source at that location. The real part of the solution is given in figure \ref{fig:wg_re}. This geometry passes the analytic solution test to $\sim\!10^{-7}$.

\begin{figure}[htbp]
    \centering
    \includegraphics[width=\textwidth]{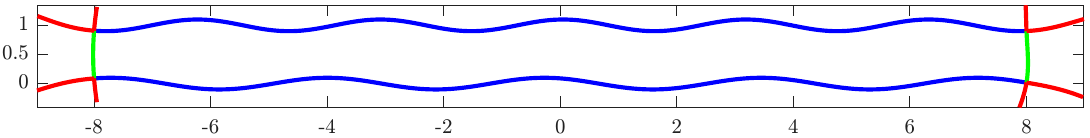}
    \caption{On the boundary of the waveguide, generalized impedance boundary conditions are used (blue), while the caps of the waveguide are given a Robin boundary condition (green). Artificial fins extend off of the corners into the exterior of the domain (red).}
    \label{fig:wg_geom}
\end{figure}
\begin{figure}[htbp]
    \centering
    \includegraphics[width=\textwidth]{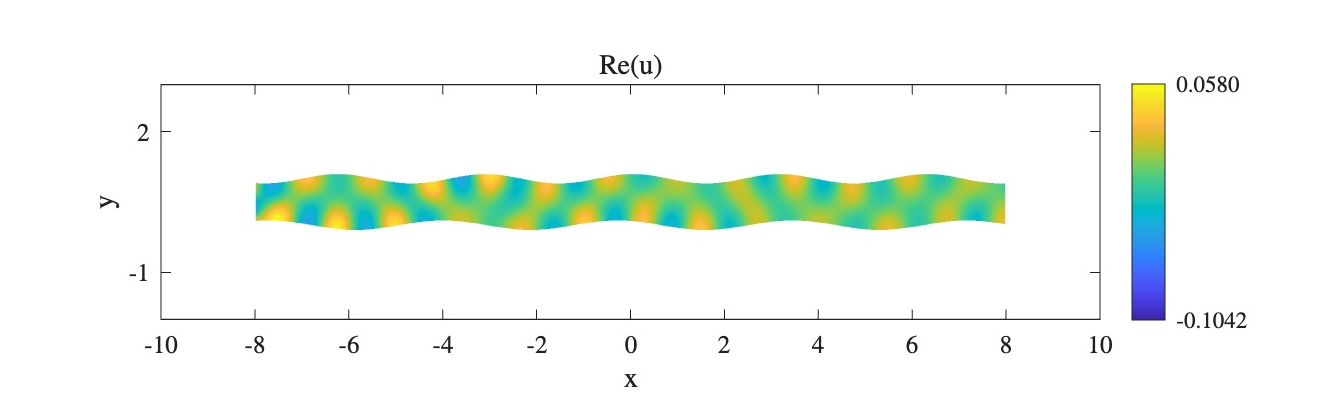}
    \caption{Real part of the solution with incoming impedance on the left boundary due to a point source.}
    \label{fig:wg_re}
\end{figure}

Next we consider the solution due purely to an inhomogeneity in the binormal boundary condition, imposed on the bottom left and top left corners of a curved waveguide. The binormal data is due to point sources at $(x,y) = (-8.1,0)$ and $(x,y) = (-8.1,0.9)$. The real part of the solution is shown in figure \ref{fig:wg_re_binorm}.

\begin{figure}[htbp]
    \centering
    \includegraphics[width=\textwidth]{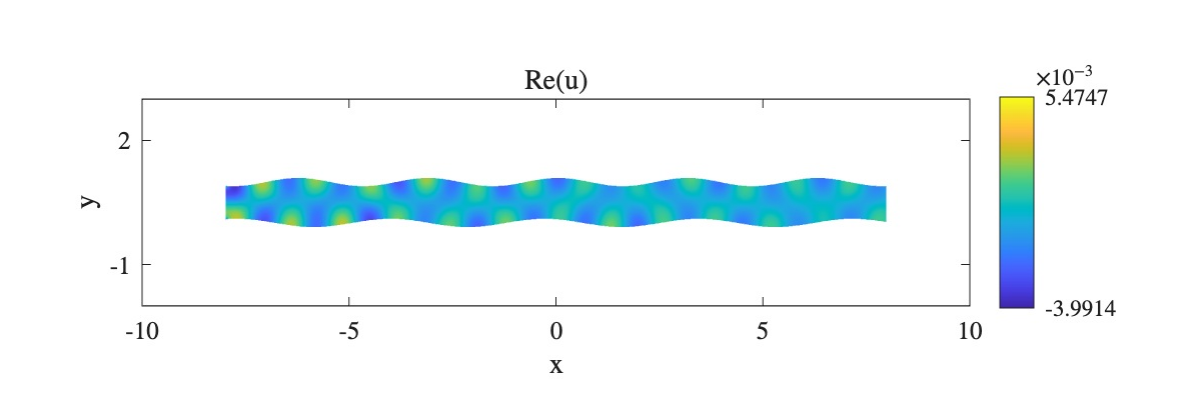}
    \caption{Real part of the solution with binormal data due to point sources near the two leftmost corners.}
    \label{fig:wg_re_binorm}
\end{figure}

In our third example for Case II, we employ impedance-to-impedance maps to decompose a complicated domain into smaller pieces (see \ref{subsec:dd}). The geometry is shown in figure \ref{fig:wg_branches_geom}. There are three waveguides with curved boundary, each connecting the triangular region on the inside to the exterior of a larger triangle. Generalized impedance boundary conditions are imposed on the boundary of the waveguide, while Neumann boundary conditions are imposed on the boundary of the triangles. The absolute value of the solution is shown in Figure \ref{fig:wg_branches_abs}. The analytic solution test for this problem is correct to $\sim\!10^{-7}$.

\begin{figure}[htbp]
    \centering
    \includegraphics[width=0.5\textwidth]{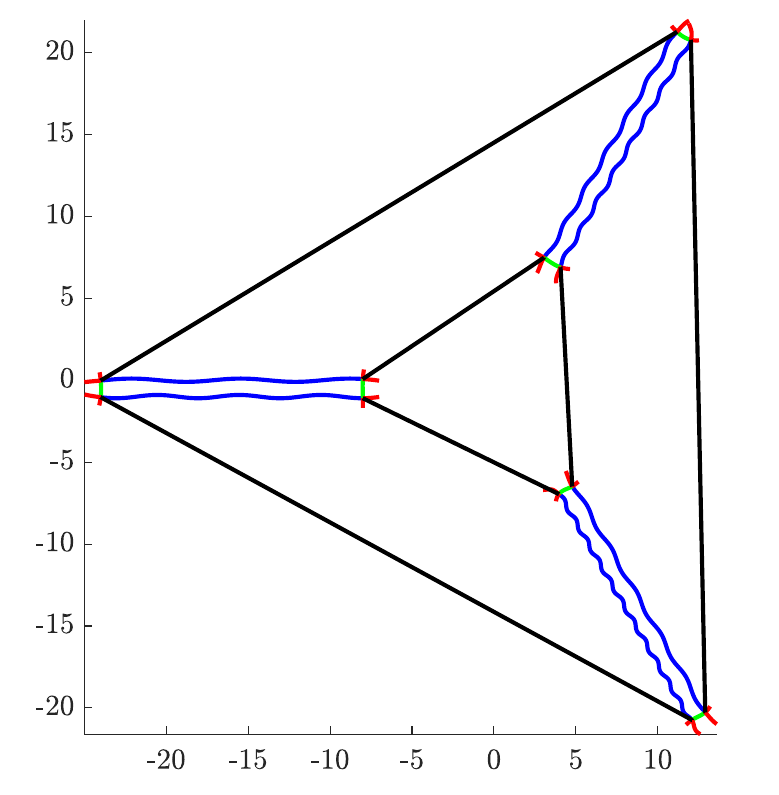}
    \caption{On the physical boundary, either Neumann (black) or generalized impedance (blue) boundary conditions are imposed. The artificial boundary, where the impedance-to-impedance matching is performed, is denoted in green. The artificial fins used for canceling corner singularities are denoted in red.}
    \label{fig:wg_branches_geom}
\end{figure}
\begin{figure}[htbp]
    \centering
    \includegraphics[width=0.5\textwidth]{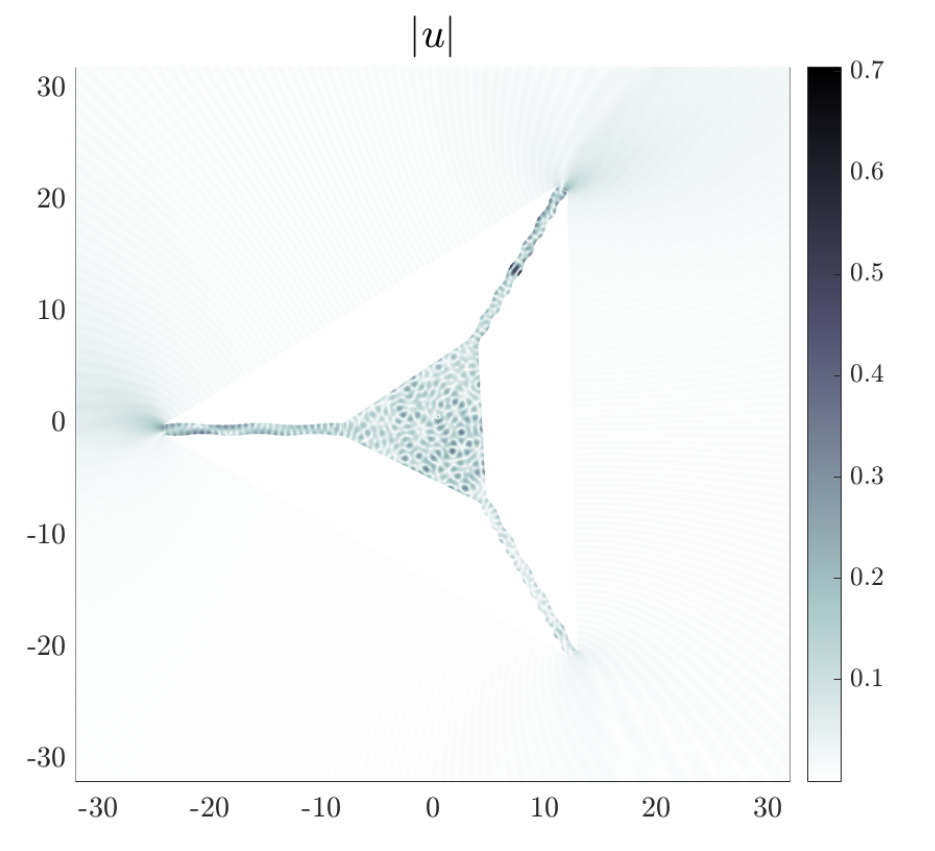}
    \caption{Magnitude of the solution due to a point source $(x,y)=(0.5,0.6)$, solved using domain decomposition.}
    \label{fig:wg_branches_abs}
\end{figure}

Our final example for Case II uses a geometry inspired by the cross-section of a phase plug. The geometry is shown in Figure \ref{fig:wg_pp_geom}, where the different sections of the physical and artificial boundaries are denoted as in Figure \ref{fig:wg_branches_geom}. The condition $\partial_nu=1$ is imposed on the rear wall of the interior of the geometry. All other boundary conditions are homogeneous. The absolute value of the solution is shown in Figure \ref{fig:wg_pp_abs}. The analytic solution test for this problem is correct to $\sim 10^{-7}$.

\begin{figure}[htbp]
    \centering
    \includegraphics[width=0.5\textwidth]{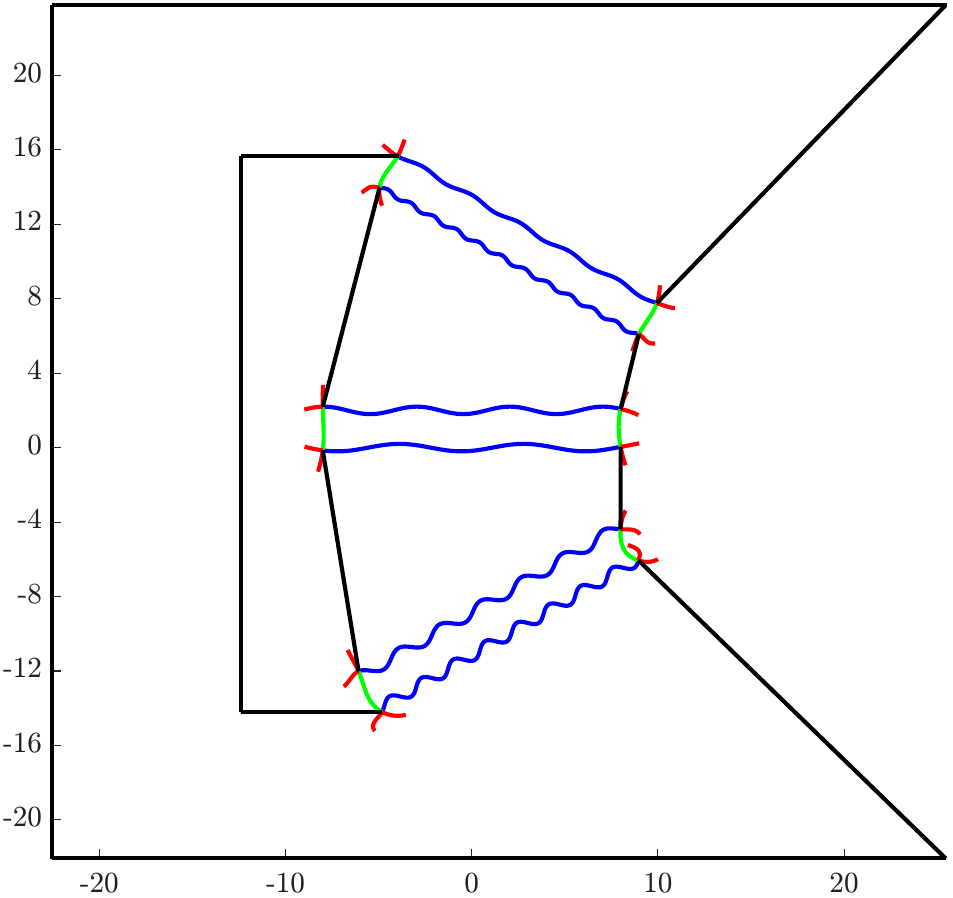}
    \caption{A geometry inspired by the cross-section of a phase plug. On the physical boundary, Neumann and generalized impedance boundary conditions are imposed on the black and blue boundaries, respectively. The artificial boundary used for impedance-to-impedance matching is denoted in green, and the fins are denoted in red.}
    \label{fig:wg_pp_geom}
\end{figure}
\begin{figure}[htbp]
    \centering
    \includegraphics[width=0.9\textwidth]{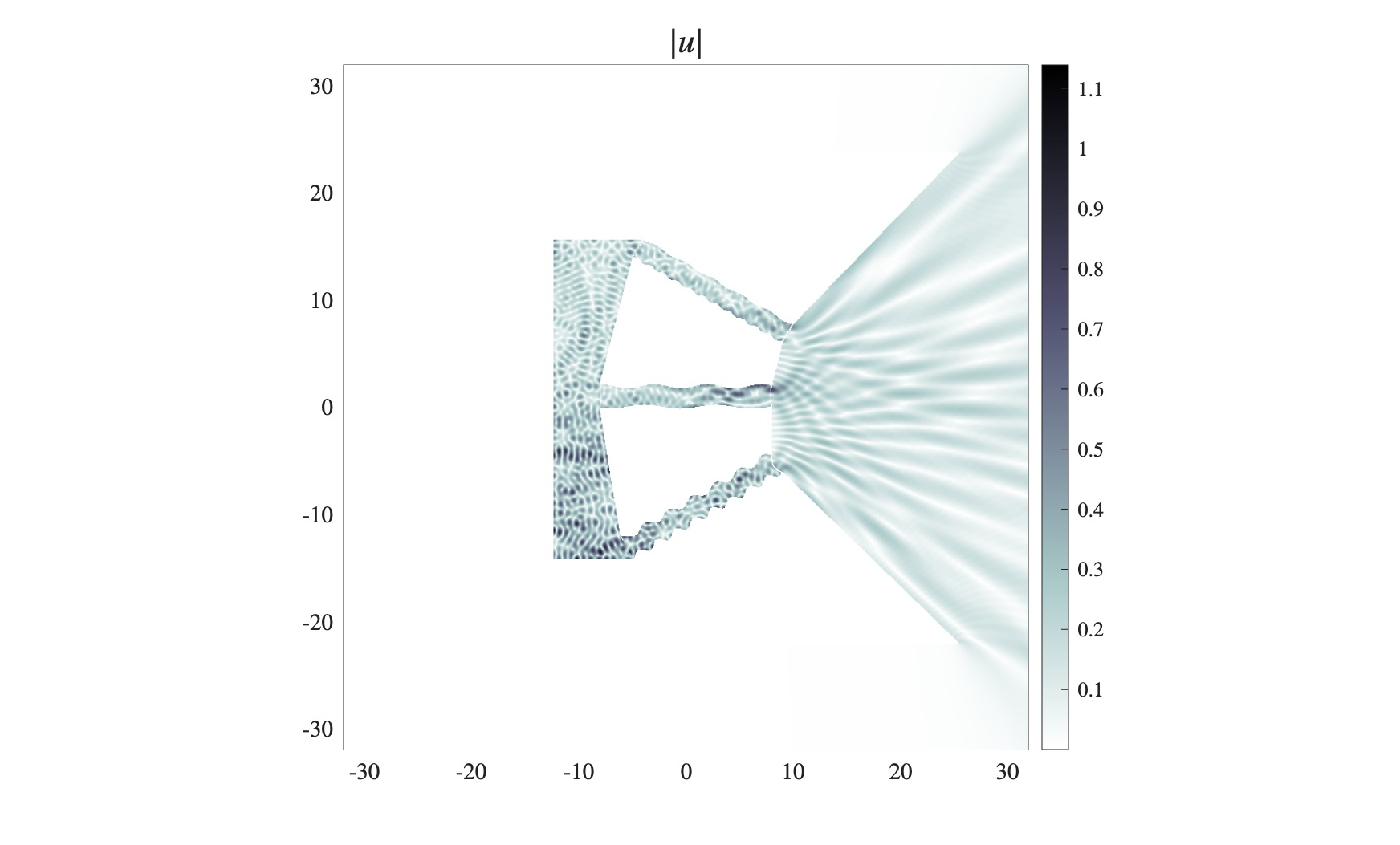}
    \caption{Magnitude of the solution due to the inhomogeneity $\partial_n u=1$ on the back wall of the interior of the device, solved using domain decomposition.}
    \label{fig:wg_pp_abs}
\end{figure}

%\FloatBarrier 

\section{Conclusion}

We present a novel integral representation for the solution of the Helmholtz equation 
with generalized impedance boundary conditions of the form \cref{eq:vtbc}. Two cases are considered: 
a smooth boundary where generalized impedance boundary conditions are imposed, and a piecewise smooth 
boundary with mixed boundary conditions. In the latter case, the impedance-to-impedance map may be used to 
link the solution in a region carrying the generalized impedance condition to a solution in a region
with other types of boundary conditions. It is shown that the integral equations obtained are 
Fredholm second-kind, in some cases up to a boundedly invertible operator. An extension of the method 
of images is introduced, which may be of independent interest, and is shown to cancel singularities 
appearing at the corners of the geometry. Numerical examples are provided to demonstrate the flexibility
and accuracy of the approach. There are some natural extensions of the work to be pursued. There is the 
three-dimensional problem, which is important for applications, and the variable coefficient
case, which is of interest in several models, e.g.~\cite{bendali1996effect,antoine2001high,aslanyurek2011generalized}.

% A novel integral representation for the solution of the Helmholtz equation with visco-thermal boundary conditions is proposed. Two cases are considered: a smooth boundary where visco-thermal boundary conditions are imposed, and a piecewise smooth boundary with mixed boundary conditions. In the latter case, the impedance-to-impedance map is used to link the solution in an acoustically narrow region to a solution in a larger region. Simulations include a geometry inspired by the cross section of a phase plug. In the first case, the integral equation we obtain is Fredholm second-kind, while in the second case, it is Fredholm second-kind up to a boundedly invertible operator. This work gives a method for efficiently simulating acoustic boundary layers to high accuracy, and is easily generalized to large and complicated geometries via the domain decomposition methods used here. Methods such as this one enable increased precision in the design and prototyping of acoustic devices. In forthcoming work, the authors will develop an integral formulation for the three dimensional version of this problem.

\section*{Acknowledgment}

J.G.H. was supported in part by a Sloan Research Fellowship. This research was supported in part by grants from the NSF (DMS-2235451) and Simons Foundation (MPS-NITMB-00005320) to the NSF-Simons National Institute for Theory and Mathematics in Biology (NITMB).
The authors thank Alex Barnett for suggesting this problem and Manas Rachh
for many useful discussions.

\begin{appendices}

\section{}\label{sec:appA}

\begin{proof}[Proof of Lemma~\ref{lem:lbDlim}.]
    With $s$ denoting arc length, the kernel of $\cD$ is of the form $\mu(s)+s^2\log(s)\nu(s)$, where $\mu$ and $\nu$ are smooth. Hence, by the dominated convergence theorem, the kernel of $\lb$ applied to $\cD$ is $\lb \del_{n_y}G$. Since the Laplace-Beltrami operator is the second arclength derivative, this implies $\lb \del_{n_y}G$ is at most logarithmically singular. 
    
    It is worth noting that since the kernel of $D$ satisfies the Helmholtz equation whenever $x\neq y$, \eqref{eq:lapdecomp} implies that
    \begin{equation}
        \lb \cD \sigma = -k^2\cD \sigma - (\kappa \cD' +\cD'' )\sigma,
    \end{equation}
    and that $\kappa \cD'+\cD''$ has logarithmically singular kernel.

    Now consider the limiting quantity for some finite $h$. Since $D\sigma$ satisfies the Helmholtz equation away from $\Gamma$, we have
    \begin{equation}\label{eq:lbD_pert}
        \lb \cD_{\tosp{\Gamma}{\Gamma_h}}\sigma = -k^2\cD_{\tosp{\Gamma}{\Gamma_h}}\sigma - \kappa_h \cD_{\tosp{\Gamma}{\Gamma_h}}'\sigma - \cD_{\tosp{\Gamma}{\Gamma_h}}'' \sigma,
    \end{equation}
    where $\kappa_h$ denotes the curvature on $\Gamma_h$. The jump of $D$ is given by \eqref{eq:Djump}, and $D'$ is continuous in the sense of \eqref{eq:D'cont}. It remains to compute the jump of $D''$. 
    
   Let $\vec{r}=x-y$ and $r=|\vec{r}|$. As $r\to0$, the Green's function satisfies the expansion
    \begin{equation}\label{eq:hank_exp} 
    \frac{i}{4} H_0^{(1)}(kr) = -\frac{1}{2\pi}\log(r)+\frac{k^2}{8\pi}r^2\log(r) + \mathcal{O}(r^4\log(r)). 
    \end{equation}
    Let $D_0$ denote the Laplace double layer operator, and define
    \begin{equation} 
    T[\sigma](x) = \int_\Gamma \left[ \partial_{n_y} r^2\log(r)\right]\sigma(y)\,{\rm d}s(y). 
    \end{equation}
    Then \eqref{eq:hank_exp} implies that
    \begin{equation}\label{eq:D''_pert}
\cD_{\Gamma\to\Gamma_h}''=(\cD''_0)_{\tosp{\Gamma}{\Gamma_h}}+\frac{k^2}{8\pi}T''_{\tosp{\Gamma}{\Gamma_h}}+M,
    \end{equation}
    where $M$ is continuous across $\Gamma$. The following limit is derived in \cite{kolm_quadruple_2003}: 
    \begin{equation}
        \lim_{h\to0} (\cD''_0)_{\tosp{\Gamma}{\Gamma_h}}\sigma = \frac{1}{2}\Delta_\Gamma\sigma +\cD_0'' \sigma.
    \end{equation}
    
    It remains to compute the jump of $T''$. Let $\gamma:(0,L)\to\mathbb{R}^2$ be an arc-length parameterization of $\Gamma$, and fix $y_0=\gamma(s_0)$. Take $x=y_0-hn_{y_0}$ and $y=\gamma(s_0+s)$. Then, by Taylor's theorem,
    \begin{equation}
        \partial_{n_x}^2\partial_{n_y}r^2\log(r) = \frac{2h^3+6hs^2}{(h^2+s^2)^2}+\zeta(h,s),
    \end{equation}
    where $\zeta=\mathcal{O}(1)$. Define $\Gamma^\epsilon = \gamma([0,s_0-\epsilon])\cup\gamma([s_0+\epsilon,L])$. Then
    \begin{align} 
    T''[\sigma](x) &= \int_{-\epsilon}^\epsilon\frac{2h^3+6hs^2}{(h^2+s^2)^2}\sigma(y(s))\,{\rm d}s\\
    &\hspace{1cm}+\int_{-\epsilon}^\epsilon \zeta(h,s)\sigma(y(s))\,{\rm d}s \nonumber\\
    &\hspace{1cm}+\int_{\Gamma^\epsilon} \left[\partial_{n_x}^2\partial_{n_y}r^2\log(r)\right]\sigma(y)\,{\rm d}s(y). \nonumber
    \end{align}
    The latter two integrals have no jump as $h\to0$. For the former we compute
    \begin{align}
        \int_{-\epsilon}^\epsilon\frac{2h^3+6hs^2}{(h^2+s^2)^2}\sigma(y(s))\,{\rm d}s &= \sigma(y_0)\int_{-\epsilon}^\epsilon\frac{2h^3+6hs^2}{(h^2+s^2)^2}\,{\rm d}s \\
        &\hspace{1cm}+ \int_{-\epsilon}^\epsilon\frac{2h^3+6hs^2}{(h^2+s^2)^2}\left[\sigma(y(s))-\sigma(y_0)\right]\,{\rm d}s. \nonumber
    \end{align}
    For $\sigma\in C^1(\Gamma)$, we have $\sigma(y(s))-\sigma(y_0)=\mathcal{O}(s)$. Since $C^1(\Gamma)\supset H^2(\Gamma)$ (\cite{aubin_nonlinear_1982}) the latter integrand is bounded as $h\to 0$, and its integral vanishes in the limit $\epsilon\to0$. The former integral is evaluated in the limit $h\to0$ as
    \begin{align}
        \sigma(y_0)\int_{-\epsilon}^\epsilon\frac{2h^3+6hs^2}{(h^2+s^2)^2}\,{\rm d}s &= \sigma(y_0)\left[8\arctan\left(\frac{\epsilon}{h}\right)-4\frac{h\epsilon}{h^2+\epsilon^2}\right] \\
        &\to 4\pi\sigma(y_0). \nonumber
    \end{align}
    This implies that
    \begin{equation}
        \lim_{h\to0}T_{\Gamma\to\Gamma_h}''\sigma = 4\pi \sigma + T'' \sigma,
    \end{equation}
    and thus, the limit of \eqref{eq:D''_pert} is 
    \begin{equation}\label{eq:D''_lim}
        \lim_{h\to0} \cD_{\Gamma\to\Gamma_h}''\sigma = \frac{1}{2}\lb \sigma +\frac{k^2}{2}\sigma +\cD''\sigma.
    \end{equation}
    To summarize, the limit of \eqref{eq:lbD_pert} is
    \begin{align}
        \lim_{h\to0} \lb \cD_{\Gamma\to\Gamma_h}\sigma &= -k^2\left(-\frac{I}{2}+\cD\right)\sigma-\kappa \cD'\sigma-\left(\frac{1}{2}\lb+\frac{k^2}{2}I+\cD''\right)\sigma \\
        &= -\frac{1}{2}\lb\sigma-k^2\cD\sigma-(\kappa \cD'+\cD'')\sigma\nonumber.
    \end{align}
\end{proof}
\begin{remark}
    We note that the limit we derive in \eqref{eq:D''_lim} differs from that derived in \cite{kolm_quadruple_2003} by a factor of 4 in the identity term.
\end{remark}

\section{}\label{sec:appB}

In this section we give proof of the regularizing effect of the addition of fins to the corners in our representation described in section~\ref{sec:mixedbcrep}. Recall from section~\ref{subsec:geomdesc} that the curves $\Gamma_*^i$ and $\Gamma_\o^j$ are assumed to meet orthogonally at each corner. Throughout, we will let $s$ denote the arc length along $\Gamma_*^i$ from the corner of interest, and $t$ denote the arc length along $\Gamma_\o^j$ from the same corner. 

We let $x(s)$ be the arclength parameterization of $\Gamma_*^i$ starting at the corner, and let $y(t)$ be the arclength parameterization of $\Gamma_\o^j$. Denote the normal vector at $x(s)$ by $n_x(s)$, the normal vector at $y(t)$ by $n_y(t)$, and analogously for the tangent vector $\tau$ and curvature $\kappa$. We assume without loss of generality that $\tau_y(0)=-n_x(0)$.

In this case, Taylor's theorem gives the following expansions on $\Gamma_*^i$:
\begin{align}
    x(s) &= -\tau_x(0)s-\frac{1}{2}\kappa_x(0)n_x(0)s^2 + \mathcal{O}(s^3), \\
    n_x(s) &= n_x(0)-\kappa_x(0)\tau_x(0)s+\mathcal{O}(s^2).
\end{align}
Reflecting $\Gamma_*^i$ across the line spanned by $n_x(0)$ therefore corresponds, up to $\mathcal{O}(s^3)$ in $x$ and $\mathcal{O}(s^2)$ in $n_x$, to the map $s\mapsto -s$. An analagous statement holds for $\Gamma_\o^j$. It will thus suffice in the computations below to consider the even and odd parts of the kernels of interest, in $s$ and $t$ separately, as appropriate.

Consider first the kernel of $\cD'_{\tosp{*_i}{\o_j}}$ in the vicinity of the corner. Due to the expansion \eqref{eq:hank_exp}, it in fact suffices to consider the kernel of the corresponding Laplace operator $(\cD_0')_{\tosp{*_i}{\o_j}}$. We write this in terms of $s$ and $t$ as
\begin{align}
    -\frac{1}{2\pi} \partial_{n_xn_y} \log r &= \frac{1}{2\pi}\left(\frac{n_x\cdot n_y}{r^2}-2\frac{(n_x\cdot \vec{r})(n_y\cdot\vec{r})}{r^4}\right) \\
    &= \frac{1}{2\pi}\frac{1}{(s^2+t^2)^3}\big(2ts^3+2t^3s-\kappa_y(0)ts^4+3\kappa_x(0)t^2s^3 \nonumber\\ 
    &\hspace{2cm}+3\kappa_y(0)t^3s^2-\kappa_x(0)t^4s+\mathcal{O}((s^2+t^2)^3)\big).\nonumber
\end{align}
Up to bounded terms, the effective kernel of the operator $(\tilde{\cD}_0')_{\tosp{*_i}{\o_j}}$, modified by the addition of even-symmetric fins, is twice the even part of the above, in $s$:
\begin{equation}
    \frac{1}{\pi}\frac{1}{(s^2+t^2)^3}\kappa_y(0)\left(-ts^4+3t^3s^2\right)+\mathcal{O}(1).
\end{equation}
Since $\kappa_y(0)=0$ by assumption (see section~\ref{subsec:geomdesc}), this implies that the kernel of $\tilde{\cD}'_{\tosp{*_i}{\o_j}}$ is bounded in the vicinity of the corner. Hence the operator $\tilde{\cD}'_{\tosp{*_i}{\o_j}}$ is compact from $H^2(\Gamma_*^i)$ to $L^2(\Gamma_\o^j)$.

The singular part of $F_i\cS_{\tosp{\o_j}{*_i}}$ arises from the evaluation of $\partial_\tau \cS_{\tosp{\o_j}{*_i}}$ at the corner. Again it suffices to consider the Laplace kernel: this is computed as
\begin{align}
    -\frac{1}{2\pi}\partial_{\tau_x}\log r &= -\frac{\tau_x\cdot\vec{r}}{r^2} \\
    &= \frac{s+\mathcal{O}(s^2+t^2)}{s^2+t^2}.\nonumber
\end{align}
Up to bounded terms, the effective kernel of the operator $\partial_\tau\tilde{\cS}_{\tosp{\o_j}{*_i}}$, modified by the addition of odd-symmetric fins, is twice the odd part of the above, in $t$. Thus its kernel is bounded in the vicinity of the corner, as desired.

\end{appendices}


\bibliography{abl}% common bib file

@article{ma_generalized_1996,
	title = {Generalized Gaussian Quadrature Rules for Systems of Arbitrary Functions},
	volume = {33},
	issn = {0036-1429, 1095-7170},
	doi = {10.1137/0733048},
	pages = {971--996},
	number = {3},
	journal = {{SIAM} Journal on Numerical Analysis},
	shortjournal = {{SIAM} J. Numer. Anal.},
	author = {Ma, J. and Rokhlin, V. and Wandzura, S.},
	date = {1996-06},
    year = {1996},
	langid = {english},
}

@article{vico_fast_2020,
	title = {A Fast Boundary Integral Method for High-Order Multiscale Mesh Generation},
	volume = {42},
	issn = {1064-8275, 1095-7197},
	doi = {10.1137/19M1290450},
	pages = {A1380--A1401},
	number = {2},
	journal = {{SIAM} Journal on Scientific Computing},
	shortjournal = {{SIAM} J. Sci. Comput.},
	author = {Vico, Felipe and Greengard, Leslie and O'Neil, Michael and Rachh, Manas},
	date = {2020-01},
    year = {2020},
	langid = {english},
}

@book{nedelec_acoustic_2001,
	location = {New York, {NY}},
	title = {Acoustic and Electromagnetic Equations: Integral Representations for Harmonic Problems},
	isbn = {978-1-4419-2889-4 978-1-4757-4393-7},
	series = {Applied Mathematical Sciences},
	shorttitle = {Acoustic and Electromagnetic Equations},
	pagetotal = {318},
	number = {144},
	publisher = {Springer},
	author = {Nédélec, Jean-Claude},
	year = {2001},
	doi = {10.1007/978-1-4757-4393-7},
    address = {New York}
}

@article{greengard_accelerating_1998,
	title = {Accelerating fast multipole methods for the Helmholtz equation at low frequencies},
	volume = {5},
	rights = {https://ieeexplore.ieee.org/Xplorehelp/downloads/license-information/{IEEE}.html},
	issn = {10709924},
	url = {http://ieeexplore.ieee.org/document/714591/},
	doi = {10.1109/99.714591},
	pages = {32--38},
	number = {3},
	journal = {{IEEE} Computational Science and Engineering},
	shortjournal = {{IEEE} Comput. Sci. Eng.},
	author = {Greengard, L. and {Jingfang Huang} and Rokhlin, V. and Wandzura, S.},
    year = {1998},
	date = {1998-09},
}

@misc{askham_chunkie_2024,
	title = {{chunkIE}: a {MATLAB} integral equation toolbox},
	url = {https://chunkie.readthedocs.io/},
	version = {1.2.0},
	author = {Askham, Travis and Rachh, Manas and O'Neil, Michael and Hoskins, Jeremy and Fortunato, Daniel and Jiang, Shidong and Fryklund, Fredrik and Goodwill, Tristan and Wang, Hai Yang and Zhu, Hai},
	date = {2024-06-08},
    year = {2024},
}

@article{kolm_quadruple_2003,
	title = {Quadruple and octuple layer potentials in two dimensions I: Analytical apparatus},
	volume = {14},
	rights = {https://www.elsevier.com/tdm/userlicense/1.0/},
	issn = {10635203},
	doi = {10.1016/S1063-5203(03)00004-6},
	shorttitle = {Quadruple and octuple layer potentials in two dimensions I},
	pages = {47--74},
	number = {1},
	journal = {Applied and Computational Harmonic Analysis},
	shortjournal = {Applied and Computational Harmonic Analysis},
	author = {Kolm, Petter and Jiang, Shidong and Rokhlin, Vladimir},
	date = {2003-01},
    year = {2003},
	langid = {english},
}

@misc{NIST:DLMF,
         key = "{DLMF}",
       title = "{NIST Digital Library of Mathematical Functions}",
howpublished = "\url{https://dlmf.nist.gov/}, Release 1.2.4 of 2025-03-15",
         url = "https://dlmf.nist.gov/",
        note = "F.~W.~J. Olver, A.~B. {Olde Daalhuis}, D.~W. Lozier, B.~I. Schneider,
                R.~F. Boisvert, C.~W. Clark, B.~R. Miller, B.~V. Saunders,
                H.~S. Cohl, and M.~A. McClain, eds.",
        year = "2025"}

@book{colton_inverse_2013,
	address = {New York, {NY}},
	edition = {3rd ed. 2013},
	title = {Inverse Acoustic and Electromagnetic Scattering Theory},
	isbn = {978-1-4614-4942-3 978-1-283-74054-8},
	series = {Applied Mathematical Sciences},
	pagetotal = {405},
	number = {93},
	publisher = {Springer},
	author = {Colton, David and Kress, Rainer},
	year = {2013},
	doi = {10.1007/978-1-4614-4942-3},
}

@book{colton_integral_2013,
	address = {Philadelphia, PA},
	title = {Integral equation methods in scattering theory},
	isbn = {978-1-61197-316-7 978-1-61197-315-0},
	series = {Classics in applied mathematics},
	pagetotal = {271},
	number = {72},
	publisher = {{SIAM}},
	author = {Colton, David and Kress, Rainer},
	editora = {{Society for Industrial and Applied Mathematics}},
	editoratype = {collaborator},
	year = {2013},
	doi = {10.1137/1.9781611973167},
}

@book{aubin_nonlinear_1982,
	address = {New York, {NY}},
	title = {Nonlinear Analysis on Manifolds. Monge-Ampère Equations},
	volume = {252},
	rights = {http://www.springer.com/tdm},
	isbn = {978-1-4612-5736-3 978-1-4612-5734-9},
	series = {Grundlehren der mathematischen Wissenschaften},
	publisher = {Springer New York},
	author = {Aubin, Thierry},
	editorb = {Artin, M. and Chern, S. S. and Doob, J. L. and Grothendieck, A. and Heinz, E. and Hirzebruch, F. and Hörmander, L. and Lane, S. Mac and Magnus, W. and Moore, C. C. and Moser, J. K. and Nagata, M. and Schmidt, W. and Scott, D. S. and Tits, J. and Van Der Waerden, B. L. and Berger, M. and Eckmann, B. and Varadhan, S. R. S.},
	editorbtype = {redactor},
	year = {1982},
	langid = {english},
	doi = {10.1007/978-1-4612-5734-9},
}

@article{aubin_espaces_1976,
	title = {Espaces de Sobolev sur les variétés Riemanniennes},
	volume = {100},
	pages = {149--173},
	journal = {Bulletin des Sciences Mathématiques},
	author = {Aubin, Thierry},
	year = {1976},
}

@article{fabes_spectra_1976,
	title = {On the spectra of a Hardy kernel},
	volume = {21},
	rights = {https://www.elsevier.com/tdm/userlicense/1.0/},
	issn = {00221236},
	doi = {10.1016/0022-1236(76)90076-8},
	pages = {187--194},
	number = {2},
	journal = {Journal of Functional Analysis},
	shortjournal = {Journal of Functional Analysis},
	author = {Fabes, E.B and Jodeit, Max and Lewis, J.E},
	year = {1976-02},
	langid = {english},
}

@article{berggren_acoustic_2018,
	title = {Acoustic boundary layers as boundary conditions},
	volume = {371},
	issn = {0021-9991},
	doi = {10.1016/j.jcp.2018.06.005},
	pages = {633--650},
	journal = {Journal of Computational Physics},
	author = {Berggren, Martin and Bernland, Anders and Noreland, Daniel},
	date = {2018-10},
    year = {2018},
	eprinttype = {arxiv},
	eprint = {1801.04177 [physics]},
}

@article{oneil_second-kind_2018,
	title = {Second-kind integral equations for the Laplace-Beltrami problem on surfaces in three dimensions},
	volume = {44},
	issn = {1019-7168, 1572-9044},
	doi = {10.1007/s10444-018-9587-7},
	pages = {1385--1409},
	number = {5},
	journal = {Advances in Computational Mathematics},
	shortjournal = {Adv Comput Math},
	author = {O’Neil, Michael},
	date = {2018-10},
    year = {2018},
	langid = {english},
}

@article{aslanyurek2011generalized,
  title={Generalized impedance boundary conditions for thin dielectric coatings with variable thickness},
  author={Aslany{\"u}rek, B and Haddar, Houssem and {\c{S}}ahint{\"u}rk, H},
  journal={Wave Motion},
  volume={48},
  number={7},
  pages={681--700},
  year={2011},
  publisher={Elsevier}
}

@article{antoine2001high,
  title={High-frequency asymptotic analysis of a dissipative transmission problem resulting in generalized impedance boundary conditions},
  author={Antoine, Xavier and Barucq, H{\'e}l{\`e}ne and Vernhet, Laurent},
  journal={Asymptotic Analysis},
  volume={26},
  number={3-4},
  pages={257--283},
  year={2001},
  publisher={SAGE Publications Sage UK: London, England}
}

@article{ammari1999generalized,
  title={Generalized impedance boundary conditions for the Maxwell equations as singular perturbations problems},
  author={Ammari, Habib and N{\'e}d{\'e}lec, J-C},
  journal={Communications in partial differential equations},
  volume={24},
  number={5-6},
  pages={24--38},
  year={1999},
  publisher={Taylor \& Francis}
}

@article{shumpert2000impedance,
  title={Impedance boundary conditions in ultrasonics},
  author={Shumpert, John D and Senior, Thomas BA},
  journal={IEEE Transactions on Antennas and Propagation},
  volume={48},
  number={10},
  pages={1653--1659},
  year={2000},
  publisher={IEEE}
}

@article{cakoni2012integral,
  title={Integral equation methods for the inverse obstacle problem with generalized impedance boundary condition},
  author={Cakoni, Fioralba and Kress, Rainer},
  journal={Inverse Problems},
  volume={29},
  number={1},
  pages={015005},
  year={2012},
  publisher={IOP Publishing}
}

@article{barnett2011new,
  title={A new integral representation for quasi-periodic scattering problems in two dimensions},
  author={Barnett, Alex and Greengard, Leslie},
  journal={BIT Numerical mathematics},
  volume={51},
  number={1},
  pages={67--90},
  year={2011},
  publisher={Springer}
}

@article{goodwill2025domain,
  title={A domain decomposition method for computing the scattering matrix of waveguide circuits},
  author={Goodwill, Tristan and Jiang, Shidong and Rachh, Manas and Sugita, Kosuke},
  journal={arXiv preprint arXiv:2509.20695},
  year={2025}
}

@article{bremer2010nonlinear,
  title={A nonlinear optimization procedure for generalized Gaussian quadratures},
  author={Bremer, James and Gimbutas, Zydrunas and Rokhlin, Vladimir},
  journal={SIAM Journal on Scientific Computing},
  volume={32},
  number={4},
  pages={1761--1788},
  year={2010},
  publisher={SIAM}
}

@article{gillman2015spectrally,
  title={A spectrally accurate direct solution technique for frequency-domain scattering problems with variable media},
  author={Gillman, Adrianna and Barnett, Alex H and Martinsson, Per-Gunnar},
  journal={BIT Numerical Mathematics},
  volume={55},
  number={1},
  pages={141--170},
  year={2015},
  publisher={Springer}
}

@Book{Atkinson1997,
author={Atkinson, Kendall E.},
title={The numerical solution of integral equations of the second kind},
series={Cambridge monographs on applied and computational mathematics; 4; 4.},
year={1997},
publisher={Cambridge University Press},
address={Cambridge},
isbn={0521583918; 9780521583916},
language={English}
}

@article{sauter2011boundary,
  title={Boundary Element Methods},
  author={Sauter, Stefan A and Schwab, Christoph},
  journal={Springer Series in Computational Mathematics},
  year={2011},
  publisher={Springer Berlin Heidelberg}
}

@book{martinsson2019fast,
  title={Fast direct solvers for elliptic PDEs},
  author={Martinsson, Per-Gunnar},
  year={2019},
  publisher={SIAM},
  address = {Philadelphia, {PA}}
}

@book{senior1995approximate,
  title={Approximate boundary conditions in electromagnetics},
  author={Senior, Thomas BA and Volakis, John Leonidas},
  number={41},
  year={1995},
  publisher={IET},
  address={London}
}

@article{durufle2006higher,
  title={Higher order generalized impedance boundary conditions in electromagnetic scattering problems},
  author={Durufl{\'e}, Marc and Haddar, Houssem and Joly, Patrick},
  journal={Comptes rendus. Physique},
  volume={7},
  number={5},
  pages={533--542},
  year={2006}
}

@article{bendali1996effect,
  title={The effect of a thin coating on the scattering of a time-harmonic wave for the Helmholtz equation},
  author={Bendali, Abderrahmane and Lemrabet, Keddour},
  journal={SIAM Journal on Applied Mathematics},
  volume={56},
  number={6},
  pages={1664--1693},
  year={1996},
  publisher={SIAM}
}

@article{haddar2005generalized,
  title={Generalized impedance boundary conditions for scattering by strongly absorbing obstacles: the scalar case},
  author={Haddar, Houssem and Joly, Patrick and Nguyen, Hoai-Minh},
  journal={Mathematical Models and Methods in Applied Sciences},
  volume={15},
  number={08},
  pages={1273--1300},
  year={2005},
  publisher={World Scientific}
}

@article{antoine2005approximation,
  author  = {Antoine, Xavier and Barucq, H{\'e}l{\`e}ne},
  title   = {Approximation by generalized impedance boundary conditions of a
             transmission problem in acoustic scattering},
  journal = {ESAIM: Mathematical Modelling and Numerical Analysis},
  volume  = {39},
  number  = {5},
  pages   = {1041--1059},
  year    = {2005},
  doi     = {10.1051/m2an:2005037}
}

@article{caubet2024ventcel,
  author  = {Caubet, Fabien and Ghantous, Joyce and Pierre, Charles},
  title   = {A priori error estimates of a {P}oisson equation with {V}entcel
             boundary conditions on curved meshes},
  journal = {SIAM Journal on Numerical Analysis},
  year    = {2024},
  doi     = {10.1137/23M1582497},
  eprint  = {2309.02437},
  archivePrefix = {arXiv}
}

@incollection{kress2018integral,
  author    = {Kress, Rainer},
  title     = {Integral equation methods in inverse obstacle scattering with a
               generalized impedance boundary condition},
  booktitle = {Contemporary Computational Mathematics --- A Celebration of the
               80th Birthday of Ian Sloan},
  editor    = {Dick, Josef and Kuo, Frances Y. and Wo{\'z}niakowski, Henryk},
  publisher = {Springer},
  address   = {Cham},
  pages     = {721--740},
  year      = {2018},
  doi       = {10.1007/978-3-319-72456-0_32}
}

@incollection{bendali1998boundary,
  author    = {Bendali, Abderrahmane},
  title     = {Boundary element solution of scattering problems relative to a
               generalized impedance boundary condition},
  booktitle = {Partial Differential Equations: Theory and Numerical Solution},
  series    = {Chapman \& Hall/CRC Research Notes in Mathematics},
  volume    = {406},
  publisher = {Chapman \& Hall/CRC},
  pages     = {10--24},
  year      = {1999},
  doi       = {10.1201/9780203744376-2}
}

@incollection{gemmrich2008boundary,
  author    = {Gemmrich, S. and Nigam, N. and Steinbach, O.},
  title     = {Boundary integral equations for the {L}aplace-{B}eltrami
               operator},
  booktitle = {Mathematics and Computation, a Contemporary View},
  series    = {Abel Symposia},
  volume    = {3},
  editor    = {Munthe-Kaas, Hans and Owren, Brynjulf},
  publisher = {Springer},
  address   = {Berlin, Heidelberg},
  pages     = {21--37},
  year      = {2008},
  doi       = {10.1007/978-3-540-68850-1_2},
  eprint    = {1111.6962},
  archivePrefix = {arXiv}
}

@article{turc2017well,
  title={Well-conditioned boundary integral equation formulations and Nystr{\"o}m discretizations for the solution of Helmholtz problems with impedance boundary conditions in two-dimensional Lipschitz domains},
  author={Turc, Catalin and Boubendir, Yassine and Riahi, Mohamed Kamel},
  journal={The Journal of Integral Equations and Applications},
  volume={29},
  number={3},
  pages={441--472},
  year={2017},
  publisher={JSTOR}
}

@article{goodwill2025parametrix,
  title={A parametrix method for elliptic surface PDEs},
  author={Goodwill, Tristan and O’Neil, Michael},
  journal={Pure and Applied Analysis},
  volume={7},
  number={1},
  pages={171--217},
  year={2025},
  publisher={Mathematical Sciences Publishers}
}

@article{banjai2022time,
  title={Time-dependent acoustic scattering from generalized impedance boundary conditions via boundary elements and convolution quadrature},
  author={Banjai, Lehel and Lubich, Christian and Nick, J{\"o}rg},
  journal={IMA Journal of Numerical Analysis},
  volume={42},
  number={1},
  pages={1--26},
  year={2022},
  publisher={Oxford University Press}
}

\end{document}